\documentclass[abstract=true]{scrartcl}

\usepackage{amsmath}
\usepackage{amsthm}
\usepackage{amssymb}
\usepackage{amsfonts}
\usepackage{hyperref}
\usepackage{color}
\usepackage{dsfont}
\usepackage{enumitem}

\usepackage{yhmath} 

\usepackage{tikz-cd}
\tikzcdset{every label/.append style = {font = \large}}

\usepackage[numbers]{natbib}

\newtheoremstyle{curs} 
  {\topsep}            
  {\topsep}            
  {\itshape}           
  {0pt}                
  {\bfseries\sffamily} 
  {.}                  
  { }                  
  {}                   
  
\newtheoremstyle{ncurs} 
  {\topsep}             
  {\topsep}             
  {}                    
  {0pt}                 
  {\bfseries\sffamily}  
  {.}                   
  { }                   
  {}                    

\theoremstyle{ncurs}
\newtheorem{dfn}{Definition}[section]
\newtheorem{rmk}[dfn]{Remark}
\newtheorem{exa}[dfn]{Example}
\newtheorem{as}[dfn]{Assumption}

\theoremstyle{curs}
\newtheorem{lem}[dfn]{Lemma}
\newtheorem{prop}[dfn]{Proposition}
\newtheorem{cor}[dfn]{Corollary}
\newtheorem{thm}[dfn]{Theorem}

\def\R{\mathbb{R}}   
\def\N{\mathbb{N}}   
\def\d{\,\mathrm{d}} 
\def\P{\mathbf{P}}   
\def\D{\mathbf{D}}   
\def\E{\mathbf{E}}   
\def\Rk{\mathcal{R}} 
\def\O{\mathcal{O}}  
\def\cN{\mathcal{N}} 
\def\e{\mathrm{e}}   
\def\Lls{L_{\mathrm{LS}}} 
\def\Lclass{L_{\mathrm{class}}}
\def\Lpin{L_{\mathrm{pin}}}
\def\Lals{L_{\mathrm{ALS}}} 
\def\Lhinge{L_{\mathrm{hinge}}}
 
\def\Cdim{C_{\mathrm{dim}}}

\newcommand{\norm}[1]{\Vert#1\Vert}

\newcommand{\id}{\mathrm{id}}

\newcommand{\supp}{\mathrm{supp}}
\newcommand{\dist}{\mathrm{dist}}

\DeclareMathOperator*{\argmin}{arg\,min}

\makeatletter
\def\namedlabel#1#2{\begingroup
    #2%
    \def\@currentlabel{#2}%
    \phantomsection\label{#1}\endgroup
}
\makeatother

\title{Adaptive Learning Rates for Support Vector Machines Working on Data with Low Intrinsic Dimension}
\author{Thomas Hamm and Ingo Steinwart}
\date {\today}

\begin{document}
\maketitle
\begin{center}
Institute for Stochastics and Applications\\
Faculty 8: Mathematics and Physics\\
University of Stuttgart\\
D-70569 Stuttgart Germany\\
\texttt{\{\href{mailto:thomas.hamm@mathematik.uni-stuttgart.de}{thomas.hamm}, \href{mailto:ingo.steinwart@mathematik.uni-stuttgart.de}{ingo.steinwart}\}@mathematik.uni-stuttgart.de}
\end{center}
\begin{abstract}
We derive improved regression and classification rates for support vector machines using Gaussian kernels under the assumption that the data has some low-dimensional intrinsic structure that is described by the box-counting dimension. Under some standard regularity assumptions for regression and classification we prove learning rates, in which the dimension of the ambient space is replaced by the box-counting dimension of the support of the data generating distribution. In the regression case our rates are in some cases minimax optimal up to logarithmic factors, whereas in the classification case our rates are minimax optimal up to logarithmic factors in a certain range of our assumptions and otherwise of the form of the best known rates. Furthermore, we show that a training validation approach for choosing the hyperparameters of an SVM in a data dependent way achieves the same rates adaptively, that is without any knowledge on the data generating distribution. 
\par\vskip\baselineskip\noindent
\end{abstract}

\section{Introduction}
Theoretical bounds on rates of convergence of nonparametric methods for regression and classification are known to be heavily influenced by the dimension of the input space, a phenomenon usually called the curse of dimensionality in statistical learning theory. Therefore, a considerable amount of effort has been put into coming up with notions of intrinsic dimension of data with the goal to improve the dependence on the dimension of the ambient space of some well-established learning rates. The apparent aim of this idea is to prove much faster rates of convergence under the assumption that the gap between the intrinsic dimension of the data and the dimension of its ambient space is large. It is safe to say, that this assumption also has a practical relevance for real life datasets considering the amount of techniques developed for dimensionality reduction. It is therefore interesting to theoretically analyze the performance of common learning methods under an assumption of low intrinsic dimension of the data.

The probably widest-spread notion to describe the intrinsic dimension of data is to assume that the data generating distribution is supported on some low-dimensional smooth submanifold $X\subset\R^d$, e.g.
\cite{BickelLi_LocalPolynomialRegressuinOnUnknownManifolds,
ScottNowak_MinimaxOptimalClassificationWithDyadicDecisionTrees,
YangDunson_BayesianManifoldRegression,
YeZhou_LearningAndApproximationByGaussiansOnRiemannianManifolds,
YeZhou_SVMLearningAndLpApproximationByGaussianOnRiemannianManifolds}.
Our notion of intrinsic dimension is instead based on the box-counting dimension of the support of the data generating distribution, which considerably generalizes the manifold assumption. A related concept of dimensionality is the so-called doubling dimension, e.g. employed in
\cite{KpotufeDasgupta_ATreeBasedRegressorThatAdaptsToIntrinsicDimension}
for tree-based regressors. While in view of regularity constraints on the input space, the doubling dimension is comparable with the box-counting dimension, its great disadvantage however is, that the value of the doubling dimension is not consistent with classical notions of dimension. So far, existing results mainly cover rather basic learning algorithms under some of the notions for low intrinsic dimension, for example local polynomial regression \cite{BickelLi_LocalPolynomialRegressuinOnUnknownManifolds}, $k$-nearest neighbor \cite{Kpotufe_kNNRegressionAdaptsToLocalIntrinsicDimension, KulkarniPosner_RatesOfConvergenceOfNearestNeighborEstimationUnderArbitrarySampling}, Nadaraya--Watson kernel regression \cite{KpotufeGarg_AdaptivityToLocalSmoothnesAndDimensionInKernelRegression}, or decision trees \cite{ScottNowak_MinimaxOptimalClassificationWithDyadicDecisionTrees}.

In this work we derive least-squares regression and classification rates for Gaussian SVMs under the assumption that the marginal of the data generating distribution on the input space $X\subset \R^d$ is supported on a set of upper box-counting dimension $\varrho\in(0,d]$ and under some standard regularity assumptions. In the regression case we assume that the target function is $\alpha$-H\"older regular for some $\alpha\in(0,\infty)$, in the context of statistical learning theory often denoted $(p,C)$-smooth in the sense of \cite[Definition 3.3]{GyorfiKohlerKrzyzakWalk_ADistributionFreeTheoryOfNonparametricRegression}. The notion of $\alpha$-H\"older regularity is described in Proposition \ref{prop:finite_Besov_norm}, for a general introduction to H\"older spaces see also \cite[Section 1.29]{AdamsFournier_SobolevSpaces}. In this setting we derive a rate of convergence of order $n^{-2\alpha/(2\alpha+\varrho)}$ up to a logarithmic factor. That is, we reproduce in some cases the well-known minimax optimal rate \cite{Stone_GlobalOptimalRatesOfConvergenceForNonparametricRegression} up to a logarithmic factor. To the best of our knowledge the only result on regression rates for SVMs under a low intrinsic dimension assumption is given in \cite{YeZhou_LearningAndApproximationByGaussiansOnRiemannianManifolds}, where the authors derive under the assumption that $X$ is a compact smooth $\varrho$-dimensional manifold and that the target function is $\alpha$-H\"older continuous for $\alpha\in(0,1]$ the learning rate $(\log^2(n)/n)^{\alpha/(8\alpha+4\varrho)}$. That is, we significantly improve the existing rates under much more general assumptions. For classification our regularity assumptions are stated in terms of margin/noise conditions. We again prove minimax optimal learning rates up to logarithmic factors for a certain range of our assumptions and otherwise reproduce the best known rates up to a logarithmic factor. As for regression, the only existing result \cite{YeZhou_SVMLearningAndLpApproximationByGaussianOnRiemannianManifolds} on Gaussian SVMs for classification is based on the assumption, that $X$ is some compact smooth manifold without boundary. However, their regularity assumptions are very restrictive as they only consider features generated by the normalized surface measure and impose a Besov-smoothness assumption on the discrete valued decision function. We discuss these issues in more detail in a comparison after Theorem \ref{thm:class_adaptive_rates}.

As the application of a Gaussian SVM requires the choice of a regularization parameter $\lambda$ and a bandwidth parameter $\gamma$, for which the asymptotically optimal choices depend on unknown characteristics of the data generating distribution, we also consider a training validation approach for choosing the hyperparameters adaptively in a data dependent way. We show, that the resulting learning method achieves the same rates for regression and classification without knowledge on the data generating distribution. Therefore, the overall conclusions of this paper can be summarized as follows: Standard, unmodified Gaussian SVMs using a training validation scheme for hyperparameter selection are, in terms of their generalization performance, adaptive to the intrinsic dimension of the data and the regularity of the target function.

Additionally, as an overall contribution, we refine the results for adaptive SVM learning rates by showing that the size of the sets of candidate values for the hyperparameters only need to grow logarithmically in the sample size instead of linear, as for example in \cite[Theorem 7.24]{SteinwartChristmannSVMs}. This gives a significant improvement on the time complexity for attaining the same rates adaptively, since this only adds a factor $\log n$ to the time complexity, instead of a factor of $n$.

The rest of this paper is organized as follows: In Section \ref{sec:preliminaries} we give a precise description of the setting we consider. The Sections \ref{sec:learning_rates_regression} and \ref{sec:learning_rates_classification} contain our main results for regression and classification, respectively. Sections \ref{sec:learning_rates_regression} and \ref{sec:learning_rates_classification} are each concluded with a comparison to existing results on learning rates under an assumption of low intrinsic dimensional data, which is comparable to ours. All proofs are deferred to the Appendix \ref{sec:appendix} including some auxiliary results on entropy numbers for Gaussian RKHSs and an oracle inequality for Gaussian SVMs for general loss functions.

\section{Preliminaries}\label{sec:preliminaries}
Let $\P$ be a distribution on $X\times Y$, where $X\subset\R^d$ is the input space and $Y\subset \R$ is the output space. We denote the marginal distribution of $\P$ on $X$ by $\P_X$. One goal of nonparametric statistics is to learn a functional relationship between $X$ and $Y$. That is, given a dataset $D=((x_1,y_1),\ldots,(x_n,y_n))\in(X\times Y)^n$ sampled from $\P^n$ we want to find a decision function $f_D:X\rightarrow Y$ such that $f_D(x)$ is a good prediction for $y$ given $x$ with high probability. To assess the quality of a decision function we fix a loss function, i.e. a measurable function $L:Y\times \R\rightarrow[0,\infty)$. In this work, we will consider the following three loss functions: The least-squares loss $\Lls(y,t):=(y-t)^2$, the classification loss $\Lclass(y,t):=\mathbf{1}_{(-\infty,0]}(y\,\mathrm{sgn}\,t)$ and because of the non-convexity\footnote{We call a loss function $L:Y\times\R\rightarrow[0,\infty)$ convex if $L(y,\cdot)$ is convex for all $y\in Y$.} of $\Lclass$, the hinge loss $\Lhinge(y,t):=\max\{0,1-yt\}$ as a surrogate. Finally, we measure the quality of a decision function by its risk
\begin{align*}
\Rk_{L,\P}(f_D):=\int_{X\times Y} L(y,f_D(x))\d\P(x,y).
\end{align*}
We further call the minimum possible risk, denoted by $\Rk_{L,\P}^*:=\inf_{f:X\rightarrow\R} \Rk_{L,\P}(f)$, the Bayes risk and a function $f_{L,\P}^*$ satisfying $\Rk_{L,\P}(f_{L,\P}^*)=\Rk_{L,\P}^*$ a Bayes decision function w.r.t. $L$ and $\P$.


Our notion of the intrinsic dimension of the data is based on the upper box-counting dimension of the support of the marginal distribution $\P_X$. To this end, recall that the support of a Borel measure $\mu$ on a subset of $\R^d$, denoted by $\supp\,\mu$, is the complement of the largest open $\mu$-zero set. To introduce the (upper) box-counting dimension we require the concept of covering numbers.
\begin{dfn}
Given a normed space $E$ and a subset $A\subset E$ we say that the points $x_1,\ldots,x_m\in E$ are an $\varepsilon$-net of $A$, if
\begin{align*}
A\subset \bigcup_{j=1}^m (x_j+\varepsilon B_E),
\end{align*}
where $B_E$ is the closed unit ball in $E$. Given an $\varepsilon>0$ the covering number $\cN_E\left(A,\varepsilon \right)$ of $A$ is defined as the minimum cardinality of an $\varepsilon$-net of $A.$ We may also write $\cN(A,\varepsilon):=\cN_E(A,\varepsilon)$, if the ambient space $E$ is known from the context. Finally, given a second normed space $F$ and a bounded, linear operator $T:E\rightarrow F$, the covering numbers of $T$ are defined by $\cN(T,\varepsilon):=\cN_F(TB_E,\varepsilon)$.
\end{dfn}
We will consider covering numbers of subsets of $\R^d$. To this end, we denote the space $\R^d$ equipped with the usual $p$-norm $\norm{\cdot}_p$ for $p\in[1,\infty]$ by $\ell_p^d$. Whenever we write $\norm{x}$ for some $x\in\R^d$ without subscript, the norm is understood to be the usual Euclidean norm. Further note that we use the symbol $\N=\{1,2,3,\ldots\}$ to denote the set of natural numbers without $0$ and $\N_0=\{0\}\cup\N$. Now, our central assumption to describe the intrinsic dimension of the data in this paper reads as follows:
\begin{center}
\begin{minipage}{0.8\textwidth}
\begin{description}
\item[\namedlabel{ass:boxdim_X}{{\textsc(DIM)}}] There exist constants $\Cdim>0$ and $\varrho>0$ such that for all $\varepsilon\in(0,1)$ we have
\begin{align*}
\cN_{\ell_\infty^d}(\supp\,\P_X,\varepsilon)\leq \Cdim\varepsilon^{-\varrho}.
\end{align*}
We further say that this assumption is satisfied exactly, if the opposite inequality is fulfilled for the same $\varrho$ and a smaller constant $c$.
\end{description}
\end{minipage}
\end{center}
Note, that \ref{ass:boxdim_X} implies that $\supp\,\P_X$ is bounded. The infimum over all $\varrho$, such that \ref{ass:boxdim_X} is fulfilled for $\varrho$ and some finite constant $\Cdim$ coincides with the so-called upper box-counting dimension of $\supp\,\P_X$, which is defined as
\begin{align}
\limsup_{\varepsilon\rightarrow 0}\dfrac{\log\cN_{\ell_\infty^d}(\supp\,\P_X,\varepsilon) }{\log \frac{1}{\varepsilon}}, \label{eqn:upper_box_counting_dim}
\end{align}
cf. \cite[Section 3.1]{Falconer_FractalGeometry}. Analogously, the lower box-counting dimension is defined by substituting $\limsup$ with $\liminf$ in (\ref{eqn:upper_box_counting_dim}) and in case those values coincide, this common limit is the box-counting dimension of $\supp\,\P_X$. Also note, that we can consider in assumption \ref{ass:boxdim_X} the covering numbers with respect to the $\ell_p^d$-norm for any $p\in[1,\infty)$, as a change of the norm will only influence the constant $\Cdim$, but not $\varrho$. One strength of assumption \ref{ass:boxdim_X} is, that the upper box-counting dimension is defined for \textit{any} bounded subset $X\subset \R^d$ and is at most $d$, that is, we basically impose no regularity in terms of smoothness on $X$ at all, in contrast to the wide-spread manifold assumption in the literature. The reason we formulate our assumption by \ref{ass:boxdim_X}, instead in terms of the upper box-counting dimension only, is that the constant $\Cdim$ influences the constants in our learning rates and we wish to track this dependence. To get some intuition on this assumption we give a few examples. For these examples assume that $\supp\,\P_X=X$.
\begin{exa}\label{exa:boxdim_cube}
Let $X=[-1,1]^d$. By a simple argument we see, that for $\varepsilon=1/m$, where $m\in\N$, we have $\cN_{\ell_\infty^d}(X,\varepsilon)=\varepsilon^{-d}$. Moreover, by \cite[Proposition 1.3.1]{CarlStephani_EntropyCompactness} there exist constants $c,C>0$, such that
\begin{align*}
c\varepsilon^{-d}\leq\cN_{\ell_\infty^d}(X,\varepsilon)\leq C\varepsilon^{-d}
\end{align*}
for all $\varepsilon\in(0,1)$. That is, \ref{ass:boxdim_X} is fulfilled exactly for $\varrho=d$. More generally we have for any bounded $X\subset \R^d$ with non-empty interior that \ref{ass:boxdim_X} is fulfilled exactly for $\varrho=d$.
\end{exa}
\begin{exa}
Let $X\subset\R^d$ be a bounded $d'$-dimensional differentiable manifold. Then \ref{ass:boxdim_X} is fulfilled for $\varrho=d'$. This follows from Example \ref{exa:boxdim_cube} and the fact, that the box-counting dimension is invariant under bi-Lipschitzian maps, cf.~\cite[Section 3.2]{Falconer_FractalGeometry}. Our assumption \ref{ass:boxdim_X} therefore includes the manifold assumption commonly used in the literature.
\end{exa}
\begin{exa}\label{exa:attractor}
The \textit{attractor} of a dynamical system is, loosely speaking, a set in the phase space of the dynamical system to which it tends to converge to, based on the initial conditions \cite{Milnor_OnTheConceptOfAttractor}. It is not unusual for such attractors of dynamical systems describing physical systems to exhibit a fractal structure, whose complexity is, amongst others, measured by their box-counting dimension \cite{FarmerOtt_TheDimensionOfChaoticAttractors}. A famous example is given by the Lorenz attractor associated to the dynamical system
\begin{align*}
x'=&-\sigma x+\sigma y \\
y'=& -xy+rx-y \\
z'=& xy-bz
\end{align*}
for real parameters $\sigma,r,b$, and was originally used to describe atmospherical convection, see \cite{Ruelle_StrangeAttractors} as well as for other examples. The Lorenz attractor is estimated to have a box-counting dimension of approximately $1.98$ for certain values of $\sigma, r, b$, see \cite{McGuinness_TheFractalDimensionOfTheLorenzAttractor}. This shows, that our assumptions allowing for non-integer dimensions is not only a mathematical quirk, but is also relevant for real-world datasets.  Suppose, for example, the feature vectors $x_i$ of the dataset $D$ are generated by observing the state of such a dynamical systems at independent, random time steps. To give a concrete example, following the discussion above, it is reasonable to assume that a dataset containing meteorological data has some low-dimensional intrinsic fractal structure, as in \cite{DeVito_AirPollution} where the authors propose an approach for estimating air pollution based on meteorological and pollution data from distant sensor stations. For further examples of fractal structures in mathematical models in physics, chemistry, and finance we refer to \cite[Chapter 18]{Falconer_FractalGeometry}.
\end{exa}
A concept related to the box-counting dimension, but with far less favorable properties is the doubling dimension. Recall that the doubling dimension of a set $X\subset \R^d$, considered as a metric space, is the smallest constant $c$, such that every ball of radius $r>0$ in $X$ can be covered by $2^c$ balls of radius $r/2$. It is used, for example, in \cite{KpotufeDasgupta_ATreeBasedRegressorThatAdaptsToIntrinsicDimension} to describe the intrinsic dimension of the data. We quickly sketch, that the box-counting dimension never exceeds the doubling dimension. To this end, let $X\subset\R^d$ be a bounded set with doubling dimension $c$. For simplicity, let us assume that $X\subset B_{\ell_2^d}$, that is $\cN_{\ell_2^2}(X,1)= 1$. By assumption $X$ can be covered by $2^c$ balls of radius $1/2$ which in turn can each be covered by $2^c$ balls of radius $1/4$. Inductively this gives us $\cN(X,1/2^k)\leq 2^{ck}$ for all $k\in\N_0$. A simple argument then gives us $\cN(X,\varepsilon)\leq 2^c\varepsilon^{-c}$ for all $\varepsilon\in(0,1)$, and hence we have $c\geq\varrho$, whenever \ref{ass:boxdim_X} is fulfilled exactly for $\varrho$. In fact, we often have $c>\varrho$. For example, as a consequence of \cite[Satz 2]{Bezdek_UeberEinigeKreisueberdeckungen} the optimal value for the doubling dimension of the unit disc $B_{\ell_2^2}$ is given by $\log_2 7\approx2.81$, while its box-counting dimension is 2. Doubling dimensions of sets, that actually have a standard notion of dimension, can rarely be computed explicitly. For example, a $\varrho$-dimensional manifold has doubling dimension $\O(\varrho)$, where the proportionality constant depends on the curvature of $X$, see e.g.~\cite[Theorem 22]{DasguptaFreund_RandomProjectionTreesAndLowDimensionalManifolds}.

Finally, we introduce our learning algorithm of interest. For a subset $X\subset \R^d$ we denote the Gaussian RKHS of width $\gamma>0$ by $H_\gamma(X)$, that is the RKHS of the kernel $k_\gamma(x,y)=\exp(-\gamma^{-2}\norm{x-y}^2)$ for $x,y\in X$.\footnote{We give a brief introduction to kernels and RKHSs in the Appendix after Corollary \ref{cor:averaged_entropy_numbers}.} Recall, that a support vector machine (SVM) with Gaussian kernel solves the regularized empirical risk minimization problem
\begin{align}
f_{D,\lambda,\gamma}:=\argmin_{f\in H_\gamma(X)} \lambda \norm{f}_{H_\gamma(X)}^2+\frac{1}{n}\sum_{i=1}^n L(y_i,f(x_i)).\label{eqn:def_SVM}
\end{align}
Note that for convex loss functions $L$ by the representer theorem \cite[Theorem 5.5]{SteinwartChristmannSVMs}, $f_{D,\lambda,\gamma}$ exists and is unique and has the expansion
\begin{align*}
f_{D,\lambda,\gamma}(x)=\sum_{i=1}^n \alpha_i k_\gamma(x,x_i),
\end{align*}
which makes the nonparametric optimization problem (\ref{eqn:def_SVM}) computable. The learning method defined by (\ref{eqn:def_SVM}) using the least-squares loss $L=\Lls$ is in the literature often called kernel ridge regression, whereas the name support vector machine is by many authors exclusively used for the learning method (\ref{eqn:def_SVM}) using the hinge loss $L=\Lhinge$.

To deal with some technical difficulties arising from the unboundedness of many loss functions, we also need to introduce the clipping operation. To this end, we say that a loss function $L$ can be clipped at some value $M>0$, if $L(y,\wideparen{t})\leq L(y,t)$ for all $y\in Y$ and $t\in\R$, where
\begin{align*}
\wideparen{t}:=
\begin{cases}
-M &\text{ for } t<-M,\\
t &\text{ for } t\in[-M,M],\\
M &\text{ for } t>M.
\end{cases}
\end{align*}
is the clipped value of $t$ at $M$, see \cite[Definition 2.22]{SteinwartChristmannSVMs}. Our final predictor is then $\wideparen{f}_{D,\lambda,\gamma}$, where the clipping operation is applied pointwise at each $x\in X$ to $f_{D,\lambda,\gamma}(x)$. If $Y\subset[-M,M]$ for some $M>0$, the least-squares loss $\Lls$ can be clipped at $M$. In the scenario of classification, our output space is given by $\{-1,1\}$. In this case the hinge loss can be clipped at 1. In Sections \ref{sec:learning_rates_regression} and \ref{sec:learning_rates_classification} let $\,\wideparen{\cdot}\,$ denote the clipping operation at the appropriate value.

The definition of $\wideparen{f}_{D,\lambda,\gamma}$ involves a choice for the regularization parameter $\lambda$ and the width $\gamma$. Usually, the asymptotically optimal choices for $\lambda$ and $\gamma$ depend on properties of the unknown distribution $\P$. To tackle this problem, we also introduce a training validation approach, that chooses (approximately) optimal hyper-parameters $\lambda$ and $\gamma$ adaptively. To this end, we split our dataset $D=((x_1,y_1),\ldots,(x_n,y_n))\in (X\times Y)^n$ into a training set $D_1:=((x_1,y_1),\ldots,(x_m,y_m))$ and a validation set $D_2:=((x_{m+1},y_{m+1}),\ldots,(x_n,y_n))$, where $m:=\lfloor n/2\rfloor+1$. Furthermore, we fix finite sets of candidate values $\Lambda_n\subset(0,1]$ and $\Gamma_n\subset (0,1]$ for $\lambda$ and $\gamma$. In the first step we then compute the SVM solutions $f_{D_1,\lambda,\gamma}$ for all $(\lambda,\gamma)\in\Lambda_n\times\Gamma_n$. In the second step choose the parameters $(\lambda_{D_2},\gamma_{D_2})$ that have the smallest empirical error on the validation set, that is
\begin{align*}
\sum_{i=m+1}^n L\left(y_i,\wideparen{f}_{D_1,\lambda_{D_2},\gamma_{D_2}}(x_i)\right)=\min_{(\lambda,\gamma)\in\Lambda_n\times\Gamma_n}\sum_{i=m+1}^n L\left( y_i,\wideparen{f}_{D_1,\lambda,\gamma}(x_i) \right).
\end{align*}
Our final estimator is then $\wideparen{f}_{D_1,\lambda_{D_2},\gamma_{D_2}}$ and we call this learning method a training validation support vector machine (TV-SVM).

\section{Learning Rates for Regression}\label{sec:learning_rates_regression}
In this section we derive learning rates for the least-squares loss function $L=\Lls$ under suitable smoothness assumptions on $f_{L,\P}^*$. First of all, recall that for the least-squares loss, the Bayes decision function is given by the conditional mean function $f_{L,\P}^*(x)=\E(Y|X=x)$, see e.g.~\cite[Section 1.1]{GyorfiKohlerKrzyzakWalk_ADistributionFreeTheoryOfNonparametricRegression}. We begin by introducing some tools for our notion of smoothness.

For a function $f:\R^d\rightarrow\R$ and $h\in\R^d$ the difference operator $\Delta_h$ is defined by $\Delta_hf(x):=f(x+h)-f(x)$. The $s$-fold application of $\Delta_h$ has the explicit expansion
\begin{align}
\Delta_h^s f(x)=\sum_{j=0}^s \binom{s}{j} (-1)^{s-j} f(x+jh).\label{eqn:difference_operator}
\end{align}
Given a measure $\mu$ on $X\subset \R^d$ we further define the $s$-th modulus of smoothness by
\begin{align}
\omega_{s,L_2(\mu)}(f,t):=\sup_{\norm{h}\leq t} \norm{\Delta_h^sf}_{L_2(\mu)}.\label{eqn:modulus_of_smoothness}
\end{align}
Finally, given an $\alpha>0$ we set $s:=\lfloor\alpha\rfloor+1$ and define the semi-norm
\begin{align}
|f|_{B_{2,\infty}^\alpha(\mu)}:=\sup_{t>0} \,t^{-\alpha}\omega_{s,L_2(\mu)}(f,t). \label{eqn:Besov_semi-norm}
\end{align}
\begin{rmk}
The so-called Besov spaces $B_{2,\infty}^\alpha(\R^d) $, commonly used in approximation theory, can be defined by means of real interpolation \cite[Chapter 3]{BerghLoefstroem_InterpolationSpaces} of Sobolev spaces, more precisely for $k_0\in\N_0$ and $k_1\in\N_0$ with $k_0\neq k_1$ and $\theta\in(0,1)$ such that $\alpha=k_0(1-\theta)+k_1\theta$
\begin{align*}
B_{2,\infty}^\alpha(\R^d):=\left(W^{k_0,2}(\R^d),W^{k_1,2}(\R^d)\right)_{\theta,\infty},
\end{align*}
 cf. \cite[Section 1.6.4]{Triebel_TheoryOfFunctionSpacesII}. Here, $W^{k,p}(\R^d)$ denotes the Sobolev space of functions on $\R^d$ that have $p$-integrable weak derivatives up to order $k$ with $p\in[1,\infty]$ and $k\in\N_0$. Interpolation theory allows one to \textit{fill the gaps} of the discrete range of smoothness of Sobolev spaces, just like the H\"older spaces fill the gaps between the spaces of $k$-times differentiable functions. In this sense, one can think of Besov spaces as Sobolev spaces with a continuous range of smoothness. For $\alpha>d/2$ we can define an equivalent norm on $B_{2,\infty}^\alpha(\R^d)$ by
\begin{align*}
\norm{\cdot}_{L_2(\R^d)}+|\cdot|_{B_{2,\infty}^\alpha(\R^d)},
\end{align*} 
where $|\cdot|_{B_{2,\infty}^\alpha(\R^d)} $ is defined with respect to the Lebesgue measure on $\R^d$ in (\ref{eqn:modulus_of_smoothness}) and (\ref{eqn:Besov_semi-norm}), cf. \cite[Section 2.6.1]{Triebel_TheoryOfFunctionSpacesII}, explaining our notation. Furthermore, we see that $|f|_{B_{2,\infty}^\alpha(\P_X)}<\infty$, whenever $f\in B_{2,\infty}^\alpha(\R^d)$ and $\P_X$ has a bounded Lebesgue-density.
\end{rmk}
In light of the remark above, we would like to have an explicit description of functions $f$ with $|f|_{B_{2,\infty}^\alpha(\P_X)}<\infty$ also for the case $\varrho<d$. The following proposition gives a sufficient condition and shows that this class of functions is indeed very rich. To formulate it, we first introduce some additional tools.

Given a function $f:\R^d\rightarrow\R$ we define, for a subset $\Omega\subset \R^d$ and $0<\beta\leq 1$, the H\"older seminorm by
\begin{align}
|f|_{\beta,\Omega}:=\sup_{\substack{x,y\in\Omega\\x\neq y}}\dfrac{|f(x)-f(y)|}{\norm{x-y}^\beta}. \label{eqn:Holder_quasinorm}
\end{align}
If $f$ is $k$-times continuously differentiable we denote for a multi-index $\nu=(\nu_1,\ldots,\nu_d)\in\N_0^d$ with $|\nu|:=\nu_1+\ldots+\nu_d=k$ the higher-order partial derivative
\begin{align*}
\partial^\nu f(x)= \frac{\partial^{|\nu|}f}{\partial x_1^{\nu_1}\ldots \partial x_d^{\nu_d}}(x).
\end{align*}
Furthermore, let $X^{+\delta}:= \{x\in\R^d:\dist(x,X)<\delta \}$ be the open $\delta$-neighborhood of $X$, where $\dist(x,X):=\inf_{y\in X}\norm{x-y}$.
\begin{prop}\label{prop:finite_Besov_norm}
Let $X\subset \R^d$ be a non-empty set and for $\alpha>0$ let $f:\R^d\rightarrow\R$ be a bounded function that is $\lfloor \alpha\rfloor$-times continuously differentiable on $X^{+\delta}$ for some $\delta>0$. Additionally, assume that one of the following conditions is satisfied:
\begin{enumerate}
\item $\alpha$ is an integer and $\partial^\nu f$ is bounded on $X^{+\delta}$ for all $|\nu|=\alpha$.
\item $\alpha$ is not an integer and for $\beta=\alpha-\lfloor\alpha\rfloor$ and all $|\nu|=\lfloor \alpha\rfloor$ we have $|\partial^\nu f|_{\beta,X^{+\delta}}<\infty$.
\end{enumerate}
Then $|f|_{B_{2,\infty}^\alpha(\mu)}<\infty$ for every measure $\mu$ with $\supp\,\mu\subset X$.
\end{prop}
In the case where $X$ has a differentiable structure it is more desirable to express our regularity assumptions on $f_{L,\P}^*$ by this intrinsic structure of $X$. Under our general assumption \ref{ass:boxdim_X} this is not possible, or at least there is no obvious way to do so. However, if $X$ is a compact $C^k$-manifold we can define the class $C^k(X)$, see \cite[Chapter 2]{Spivak_AComprehensiveIntroductionToDifferentialGeometry}. Moreover, by Whitney's extension theorem \cite[Theorem 2.3.6]{Hormander_TheAnalysisOfLinearPartialDifferentialOperatorsI} every function $f\in C^k(X)$ has an extension to a function $\tilde{f}\in C^k(X^{+\delta})$ for some $\delta>0$, eliminating the need to define our regularity assumptions by means of the ambient space in these cases. Especially this means, that the assumption $f_{L,\P}^*\in C^k(X)$, as in \cite{YangDunson_BayesianManifoldRegression}, is not more general than our assumption $ f_{L,\P}^*\in C^k(X^{+\delta}) $.

The following theorem establishes an oracle inequality for Gaussian SVMs under assumption \ref{ass:boxdim_X}.
\begin{thm}\label{thm:ls_oracle_inequality}
Assume $\P$ satisfies \ref{ass:boxdim_X} with parameters $\Cdim, \varrho$ and that $Y\subset [-M,M]$. Further assume that $f_{L,\P}^*\in L_2(\R^d)\cap L_\infty(\R^d)$ as well as $|f_{L,\P}^*|_{B_{2,\infty}^\alpha(\P_X)}<\infty$. Then for all $\tau>0$, $n>1$, $\lambda\in(0,1)$ and $\gamma\in(0,1)$ we have
\begin{align}\label{eqn:ls_oracle_inequality}
\begin{split}
\Rk_{L,\P}(\wideparen{f}_{D,\lambda,\gamma})-\Rk_{L,\P}^* \leq& c_1\norm{f_{L,\P}^*}_{L_2(\R^d)}^2\lambda\gamma^{-d}+c_2  |f_{L,\P}^*|_{B_{2,\infty}^\alpha(\P_X)}^2\gamma^{2\alpha} \\
& +c_3 K \lambda^{-1/\log n}\gamma^{-\varrho}n^{-1}\log^{d+1}n+c_4\frac{\tau}{n}
\end{split}
\end{align}
with probability $\P^n$ not less than $1-3\e^{-\tau}$, where $c_1=9\pi^{-d/2}4^s$,
\begin{align*}
c_2=9\left( \frac{ \Gamma\left(\frac{\alpha+d}{2} \right) }{\Gamma\left( \frac{d}{2}\right)} \right)^2 2^{-\alpha}, \quad c_3=\max\left\{16M^2,1\right\}\max\left\{\Cdim,4M^2\right\}, \\
 c_4=3456M^2+15\max\{(2^s\norm{f_{L,\P}^*}_{L_\infty(\R^d)}+M)^2,4M^2\}
\end{align*}
with $s=\lfloor \alpha\rfloor+1$ and $K$ is a constant independent of $\P,n,\lambda$ and $\gamma$.

If in addition $\supp\,\P_X$ is contained in a $\varrho$-dimensional, infinitely differentiable, compact manifold then the factor $\log^{d+1}n$ on the right hand side of (\ref{eqn:ls_oracle_inequality}) can be substituted by $\log^{\varrho+1}n$, where the constant $c_3$ then additionally depends on the curvature of the manifold.
\end{thm}
Using the theorem above we can easily derive learning rates by choosing specific values for the regularization parameter $\lambda$ and the bandwidth $\gamma$.
\begin{cor}\label{cor:LS_learning_rates}
Let the assumptions of Theorem \ref{thm:ls_oracle_inequality} be satisfied with $\norm{f_{L,\P}^*}_{L_2(\R^d)}\leq C_1, \norm{f_{L,\P}^*}_{L_\infty(\R^d)}\leq C_2$ and $|f_{L,\P}^*|_{B_{2,\infty}^\alpha(\P_X)}\leq C_3$. Choosing $\gamma_n=n^{-1/(2\alpha+\varrho)} $ and $ \lambda_n=n^{-b}$ for some $b\geq (2\alpha+d)/(2\alpha+\varrho)$ there then exists a constant $C>0$ only depending on $\Cdim,C_{1,2,3}$ and $M$ such that for all $n>1$ and $\tau\geq1$ we have
\begin{align*}
\Rk_{L,\P}(\wideparen{f}_{D,\lambda_n,\gamma_n})-\Rk_{L,\P}^*\leq C\tau\, n^{-\frac{2\alpha}{2\alpha+\varrho}}\log^{d+1}n
\end{align*}
with probability $\P^n$ not less than $1-\e^{-\tau}$.

If in addition $\supp\,\P_X$ is contained in a $\varrho$-dimensional, infinitely differentiable, compact manifold then the factor $\log^{d+1}n$ above can be substituted by $\log^{\varrho+1}n$, where the constant $C$ then additionally depends on the curvature of the manifold.
\end{cor}
\begin{rmk}
The boundedness assumption $Y\subset [-M,M]$ can be relaxed to an exponential decay of the distribution of the noise variable $y-f_{L,\P}^*(x)$. Under this assumption one can show, that by choosing a sequence of logarithmically growing clipping values $M=M_n$, the resulting learning algorithm achieves the same rate as in the corollary above modulo an extra $\log^2 n$ factor. This statement can be proven by showing that the clipping value $M_n$ is correct with sufficiently high probability. As the details are a little bit technical and would distract from the actual conclusions of our results, we refer to \cite[Theorem 3.6]{EbertsSteinwart_OptimalRegressionRatesForSVMsUsingGaussianKernels}.
\end{rmk}
\begin{rmk}\label{rmk:optimality_reg_rates}
We briefly want to discuss the optimality of the rate in Corollary \ref{cor:LS_learning_rates}. The classical result of Stone \cite{Stone_GlobalOptimalRatesOfConvergenceForNonparametricRegression} considers the case where $\P_X$ is the uniform distribution on $[0,1]^d$ and states that $n^{-\frac{2\alpha}{2\alpha+d}}$ is the optimal rate of convergence. This statement can directly be generalized to the case where $\P_X$ is the uniform $d'$-dimensional distribution on the cube in the first $d'$ axes of $\R^d$, where $d'\in\{1,\ldots,d\}$. Note that in this case $\P$ satisfies \ref{ass:boxdim_X} for all $\varrho\geq d'$. From this we can conclude that in the case $\varrho\in\{1,\ldots,d\}$ the rates in Corollary \ref{cor:LS_learning_rates} are optimal up to the logarithmic factor. In the general non-integer case $\varrho\geq 1$ we can still deduce a lower bound of order $2\alpha/(2\alpha+\lfloor\varrho\rfloor)$, however there is no immediate argument that this is the optimal lower bound. We strongly hypothesize that the general optimal lower bound is of order $2\alpha/(2\alpha+\varrho)$ but we will leave this as a conjecture for possible future research.
\end{rmk}
\begin{rmk}\label{rmk:alternative_reg_ass_reg}
It is also possible to formulate Theorem \ref{thm:ls_oracle_inequality} and Corollary \ref{cor:LS_learning_rates} under alternative regularity assumptions. To briefly elaborate this, recall that in \cite{YeZhou_SVMLearningAndLpApproximationByGaussianOnRiemannianManifolds} the authors consider the case where $X$ is a compact, connected and smooth submanifold of $\R^d$ without boundary and consider a convolution-type operator $S_{\gamma}:L_2(\mu)\rightarrow H_\gamma(X)$, where $\mu$ is the measure on $X$ defined by the Riemannian volume form and derive the bounds
\begin{align}
\norm{S_{\gamma}f}_{H_\gamma(X)}^2\leq C_1 \norm{f}_{L_2(\mu)}^2\gamma^{-d} \label{eqn:zhou_reg_term} \\
\norm{S_{\gamma}f-f}_{L_2(\mu)}^2\leq C_2\norm{f}_{W^{2,2}(X)}^2\gamma^4 \label{eqn:zhou_approx_error}\\
\norm{S_\gamma f}_{\ell_\infty(X)} \leq C_3 \norm{f}_{\ell_\infty(X)} \label{eqn:zhou_sup_bound}
\end{align}
\cite[Lemma 4, Theorem 2, and Lemma 2]{YeZhou_SVMLearningAndLpApproximationByGaussianOnRiemannianManifolds}, where $W^{2,2}(X)$ denotes the Sobolev space on $X$. Using these results one can easily derive a modification of Theorem \ref{thm:ls_oracle_inequality} and Corollary \ref{cor:LS_learning_rates} under the assumption that $f_{L,\P}^*\in W^{2,2}(X)$ is bounded and $\P_X$ has a bounded density with respect to $\mu$ and prove learning rates of the form $n^{-\frac{4}{4+\varrho}}\log^{d+1}n$. This is done by simply using (\ref{eqn:zhou_approx_error}) instead of Lemma \ref{lem:L2error}, (\ref{eqn:zhou_reg_term}) instead of Lemma \ref{lem:regularization_term} and the supremum bound (\ref{eqn:zhou_sup_bound}) instead of an analogous bound we derive in the proof of Theorem \ref{thm:ls_oracle_inequality}. Unfortunately, the authors also point out, that the order of approximation cannot be improved if we assume $f\in W^{m,2}(X)$ for some $m>2$ using the operator $S_\gamma$.
\end{rmk}
The learning rates in Corollary \ref{cor:LS_learning_rates} can only be achieved, if we know the intrinsic dimension $\varrho$ of the data, as well as the regularity $\alpha$ of the Bayes decision function. However, this is highly unrealistic in practice. The following theorem therefore shows, that a TV-SVM with appropriately chosen candidate sets $\Lambda_n$ and $\Gamma_n$ achieves the same rate without knowledge on $\varrho$ and $\alpha$.
\begin{thm}\label{thm:LS_adaptive_rates}
Let $A_n$ be a minimal $1/\log n$-net of $(0,1]$ with $1\in A_n$ and let $B_n$ be a minimal $1/\log n$-net of $[1,d]$ with $d\in B_n$. Set $\Gamma_n:=\{n^{-a}:a\in A_n\}$ and $\Lambda_n:=\{n^{-b}:b\in B_n\}$. Let the assumptions of Theorem \ref{thm:ls_oracle_inequality} be satisfied with $\norm{f_{L,\P}^*}_{L_2(\R^d)}\leq C_1, \norm{f_{L,\P}^*}_{L_\infty(\R^d)}\leq C_2$ and $|f_{L,\P}^*|_{B_{2,\infty}^\alpha(\P_X)}\leq C_3$ and assume $\varrho\geq 1$. Then there exists a constant $C>0$ only depending on $\Cdim,C_{1,2,3}$ and $M$ such that for all $n>1$ and $\tau\geq 1$ the TV-SVM using $\Lambda_n$ and $\Gamma_n$ satisfies
\begin{align*}
\Rk_{L,\P}(\wideparen{f}_{D_1,\lambda_{D_2},\gamma_{D_2}})-\Rk_{L,\P}^*\leq C\tau\, n^{-\frac{2\alpha}{2\alpha+\varrho}}\log^{d+1}n
\end{align*}
with probability $\P^n$ not less than $1-\e^{-\tau}$.

If in addition $\supp\,\P_X$ is contained in a $\varrho$-dimensional, infinitely differentiable, compact manifold then the factor $\log^{d+1}n$ above can be substituted by $\log^{\varrho+1}n$, where the constant $C$ then additionally depends on the curvature of the manifold.
\end{thm}
\begin{rmk}\label{rmk:validation_set}
The proof of Theorem \ref{thm:LS_adaptive_rates} shows that the statement also holds if we pick as candidate set for $\lambda$ the singleton $\Lambda_n=\{ n^{-d} \}$. We decided to formulate the theorem as it is, because this choice is closer to the practical usage of the training validation approach. To see why a singleton $\Lambda_n$ is sufficient recall that to achieve optimal rates the regularization parameter $\lambda_n$ only needs to satisfy $\lambda_n=n^{-b}$ for some $b\geq (2\alpha+d)/(2\alpha+\varrho)$. If we assume $\varrho\geq 1$ this bound on $b$ is satisfied for $b=d$. This also shows that with no lower bound on $\varrho$ the regularization parameter $\lambda_n$ possibly needs to decay arbitrarily fast as $(2\alpha+d)/(2\alpha+\varrho)$ is unbounded for $\alpha,\varrho>0$. On the one hand this shows why we need the additional constraint $\varrho\geq 1$ in Theorem \ref{thm:LS_adaptive_rates}, on the other hand we want to mention that the case $\varrho<1$ is of little practical interest anyway. Nevertheless, the statement of Theorem \ref{thm:LS_adaptive_rates} still holds with the rate $n^{-2\alpha/(2\alpha+1)}\log^{d+1}n$ in the case $\varrho<1$.
\end{rmk}
We conclude this section with a comparison to existing results on learning rates for least-squares regression under the assumption, that the data has some low dimensional intrinsic structure, which is described by imposing constraints on $\supp\,\P_X$.
\begin{description}
\item[SVMs.] In \cite{YeZhou_LearningAndApproximationByGaussiansOnRiemannianManifolds} the authors consider the case, where the input space $X\subset \R^d$ is a compact, smooth manifold of dimension $\varrho$ and that the Bayes decision function is $\alpha$-H\"older continuous with respect to the geodesic distance on $X$ for some $\alpha\in(0,1]$. They derive learning rates of the form $ (\log^2(n)/n)^{\alpha/(8\alpha+4\varrho)}$. Under these assumptions on $X$, the assumption that the Bayes function is $\alpha$-H\"older continuous with respect to the geodesic distance is equivalent to $\alpha$-H\"older continuity with respect to the Euclidean distance, which is a consequence of \cite[Lemma 1]{YeZhou_LearningAndApproximationByGaussiansOnRiemannianManifolds}. That is, Corollary \ref{cor:LS_learning_rates} gives a significant improvement of the result in \cite{YeZhou_LearningAndApproximationByGaussiansOnRiemannianManifolds} under much less restrictive assumptions. Additionally, they do not consider adaptive parameter selection, i.e. the dimension $\varrho$ of $X$ and $\alpha$ need to be known.
\item[Bayesian Regression with Gaussian processes.] In \cite{YangDunson_BayesianManifoldRegression} the authors consider the case of a compact $\varrho$-dimensional manifold $X\subset \R^d$, which is sufficiently regular and a $C^\alpha$ Bayes function, where $\alpha\leq 2$. For Bayesian regression using Gaussian processes with squared exponential covariance they derive a rate of $n^{-2\alpha/(2\alpha+\varrho)}\log^{\varrho+1} n$, which is identical to ours, but under much more restrictive assumptions on both $\alpha$ and $\supp\,\P_X$. Additionally, they present a training validation scheme for choosing the hyperparameters of the prior distribution. Under some additional technical assumption, which is hard to verify, they prove that this method achieves the same rates adaptively. Also note that Bayesian GP regression is related to Gaussian least-squares SVMs in the sense that the posterior mean function of the Gaussian process coincides with the SVM solution provided the prior distribution is suitably chosen, see \cite[Proposition 3.6]{KanHenSejSri18}. The choice of the prior distribution in \cite{YangDunson_BayesianManifoldRegression} however does not lead to the same decision function as the one chosen by a Gaussian least-squares SVM.
\item[Local Polynomial Regression.] In \cite{BickelLi_LocalPolynomialRegressuinOnUnknownManifolds} the authors consider local linear regression in the setting of a differentiable $\varrho$-dimensional manifold $X\subset\R^d$ and a twice differentiable Bayes function. They further assume, that $\P_X$ has a differentiable density w.r.t. local charts and prove the learning rate $n^{-4/(4+\varrho)}$. They also state, that the result can be extended if the Bayes function is $\alpha$-times differentiable using polynomials of degree $\alpha-1$ to achieve the rate $n^{-2\alpha/(2\alpha+\varrho)}$. They propose a training validation scheme for bandwidth selection, but give no theoretical guarantees.
\item[Tree-Based Regressor.] In \cite{KpotufeDasgupta_ATreeBasedRegressorThatAdaptsToIntrinsicDimension} the authors consider a tree-based locally constant regressor. Based on the doubling dimension\footnote{See also the short discussion after Example \ref{exa:attractor}.} $\varrho$ of the input space $X$ they prove that for a Lipschitz continuous target function their estimator achieves the learning rate $n^{-2/(2+k)}$ for a constant $k\in\O(\varrho\log \varrho)$. They achieve this rate adaptively using a training validation scheme, as well as a stopping criterion for building the tree.
\item[$k$-Nearest Neighbor.] In \cite{KulkarniPosner_RatesOfConvergenceOfNearestNeighborEstimationUnderArbitrarySampling} the authors show, that under assumption \ref{ass:boxdim_X} and for an $\alpha$-H\"older continuous target function, where $0<\alpha \leq 1$, the $k$-NN rule achieves a learning rate of $n^{-2\alpha/(2\alpha+\varrho)}$. They only achieve this rate with knowledge on both, $\varrho$ and $\alpha$.
\end{description}

In the comparison above we only mentioned results, where the assumption on the intrinsic dimension is directly comparable to ours. For the sake of completeness we want to mention, that there are results, where the intrinsic dimension of the data is described by assuming that $\P_X$ is a so-called doubling measure. For example, this assumption is considered in \cite{Kpotufe_kNNRegressionAdaptsToLocalIntrinsicDimension} for the $k$-nearest neighbor method and in \cite{KpotufeGarg_AdaptivityToLocalSmoothnesAndDimensionInKernelRegression} for Nadaraya--Watson kernel regression estimators. For an introduction to this notion of intrinsic dimensionality and the results built on it, we refer the reader to the cited articles.

As Remark \ref{rmk:alternative_reg_ass_reg} already indicates, some of our later results make it very easy to derive learning rates under slightly different regularity assumptions on the Bayes decision function under assumption \ref{ass:boxdim_X}. Even more general, we derive an oracle inequality in Theorem \ref{thm:general_oracle_inequality} for SVMs using Gaussian kernels under assumption \ref{ass:boxdim_X} for generic loss functions $L$. Using Theorem \ref{thm:general_oracle_inequality} it will then be an easy task to derive improved learning rates in the style of Corollary \ref{cor:LS_learning_rates} for other regression tasks. For example, the conditional $\tau$-quantile function can be estimated using the pinball loss
\begin{align*}
\Lpin(y,t):=
\begin{cases}
(1-\tau)(t-y), &\text{ if } y<t,\\
\tau(y-t), &\text{ if } t\geq y.
\end{cases}
\end{align*}
Using a variance bound, see Definition \ref{dfn:variance_bound}, and a calibration inequality  for the pinball loss from \cite{SteinwartChristmann_EstimatingConditionalQuantilesWithTheHelpOfThePinballLoss} and imposing some standard regularity assumptions on $f_{\Lpin,\P}^*$ we can thus derive learning rates for the pinball loss, as well as bounds on the $L_q(\P_X)$ distance from $f_{D,\lambda,\gamma}$ to $f_{\Lpin,\P}^*$ under assumption \ref{ass:boxdim_X}. This would generalize some of the results from \cite{EbertsSteinwart_OptimalRegressionRatesForSVMsUsingGaussianKernels} in the sense, that we can substitute $d$ with $\varrho$ in their learning rates. The same is true for the conditional expectile function, which is estimated by the asymmetric least-squares loss
\begin{align*}
\Lals(y,t):=
\begin{cases}
(1-\tau)(t-y)^2, &\text{ if } y<t,\\
\tau(y-t)^2, &\text{ if } t\geq y,
\end{cases}
\end{align*}
where $\tau\in(0,1)$, since in \cite{FarooqSteinwart_LearningRatesforKernelBasedExpectileRegression} the necessary variance bound and a calibration inequality for $\Lals$ are derived.

Another possibility to further generalize our results is to consider a larger class of kernels. For example Corollary \ref{cor:averaged_entropy_numbers}, a main tool for bounding the statistical error, can be generalized with the same techniques to anisotropic Gaussian kernels, i.e.~Gaussian kernels that have a different bandwidth parameter in each covariate. Anisotropic Gaussian SVMs are for example considered in \cite{HangSteinwartAnisotropic} where the authors show that, compared to a regular Gaussian kernel, the anisotropic one has an improved performance for regression functions that have a varying degree of smoothness in each covariate.

\section{Learning Rates for Classification}\label{sec:learning_rates_classification}
Throughout this section let $Y=\{-1,1\}$ and fix a version $\eta:X\rightarrow[0,1]$ of the posterior probability of $\P$, that is the probability measures $\P(\,\cdot\,|x)$ on $Y$ defined by $\P(\{1\}|x)=\eta(x)$ for $x\in X$ fulfill
\begin{align*}
\P(A\times B)=\int_A \P(B|x) \d\P_X(x)
\end{align*}
for all measurable sets $A\subset X$ and $B\subset Y$. With the posterior probability $\eta$ any
\emph{$\R$-valued }optimal labeling strategy $f_{\Lclass,\P}^*:X\to \R$ satisfies $\mathrm{sign}( f_{\Lclass,\P}^*(x))=\mathrm{sign}(2\eta(x)-1)$, cf. \cite[Example 2.4]{SteinwartChristmannSVMs}. Regularity assumptions on $\P$ for classification can be described by smoothness of $\eta$. However, can be rather restrictive, since irregularities in a region, where $\eta$ is close to 1 do not make the classification task harder. Additionally, smoothness assumptions on $\eta$ are often combined with the assumption, that $\P_X$ has a Lebesgue-density. However, these types of assumptions are typically motivated by considering plug-in classification rules, which implicitly treat the binary classification problem as a least-squares regression problem. Now recall that, in the binary classification setting, the Bayes decision function $f_{\Lls,\P}^*$ for the least-squares loss  equals $2\eta-1$, and therefore
smoothness properties of $\eta$ directly translate to such for $f_{\Lls,\P}^*$. 
If, however, the hinge loss is used, the Bayes decision function equals 
$\mathrm{sign}(2\eta -1)$ on $\tilde X:= \{x\in X: \eta(x) \not\in\{0,1/2,1\}\}$, and any hinge loss consistent learning method is required to approach this
Bayes decision function on $\tilde X$. Now observe that the smoothness of $\eta$ has 
little influence on the smoothness of $\mathrm{sign}(2\eta -1)$, in fact, 
the latter is typically a step function, even if $\eta$ is a $C^\infty$-function.
For this reason, smoothness assumptions for classification algorithms using the hinge loss
seem to be less attractive.

To avoid these issues, we first observe that it is intuitively hard to correctly predict a label $y$ for $x\in X$ with $\eta(x)\approx 1/2$. Similarly, it may be harder to make a correct prediction for points $x\in X$ that are located near the decision boundary $\{\eta = 1/2\}$. The following assumptions capture this intuition by imposing constraints on the mass and/or location of the critical level $\eta(x)\approx 1/2$.

\begin{as}\label{as:TNC}
There exist constants $C_*>0$ and $q\in[0,\infty]$ such that
\begin{align*}
\P_X\big(\{x\in X: |2\eta(x)-1|<t\}\big)\leq (C_* t)^q
\end{align*}
for all $t\geq 0$.
\end{as}
Note that every distribution satisfies Assumption \ref{as:TNC} for $q=0$. The other extreme case $q=\infty$, using the convention $t^\infty = 0$ for $t\in(0,1)$, corresponds to the high regularity case where $\eta(x)$ is bounded away from the critical level $1/2$ for $\P_X$-almost all $x\in X$. Assumption 4.1 is in the literature often referred to as Tsybakov noise condition and is a standard assumption in nonparametric classification. It was originally introduced in \cite{MammenTsybakov_SmoothDiscriminationAnalysis} for discrimination analysis and later in \cite{Tsybakov_OptimalAggregationOfClassifiersInStatisticalLearning} for classification. Note that Assumption \ref{as:TNC} only restricts the mass and not the location of noise. Our second regularity assumption relates the noise to the distance to the decision boundary. To this end, we need the following
\begin{dfn}
Let $X_{-1}:=\{x\in X:\eta(x)<1/2\}$ and $X_1:=\{x\in X:\eta(x)>1/2\}$ and define
\begin{align*}
\Delta(x):=\begin{cases}
\dist(x,X_1) &\text{ if } x\in X_{-1},\\
\dist(x,X_{-1}) &\text{ if } x\in X_1,\\
0&\text{ else,}
\end{cases}
\end{align*}
where $\dist(x,A):=\inf_{y\in A}\norm{x-y}$.
\end{dfn}

\begin{as}\label{as:MNE}
There exist constants $C_{**}>0$ and $\beta >0$ such that
\begin{align*}
\int_{\{x\in X:\Delta(x)<t\}}|2\eta(x)-1|\d\P_X(x)\leq C_{**}t^\beta
\end{align*}
for all $t\geq 0$.
\end{as}
Note that neither of our Assumptions \ref{as:TNC} and \ref{as:MNE} require the existence of a density of $\P_X$ and are therefore perfectly suited to be combined with our assumption \ref{ass:boxdim_X}. The condition in Assumption \ref{as:MNE} was introduced (in a slightly different version) in \cite{SteinwartScovel_FastRatesForSupportVectorMachinesUsingGaussianKernels} to prove fast classification rates for Gaussian SVMs. The authors used the term geometric noise condition to describe this assumption. Assumption \ref{as:MNE} was further adopted in \cite{BlaschzykSteinwart_ImprovedClassificationRatesUnderRefinedMarginConditions} for the analysis of a histogram based classifier. Also, in \cite{UrnerBenDavid_ProbabilisticLipschitzness} the authors point out that their \textit{probabilistic Lipschitzness} condition in Definition 1 is closely related to Assumption \ref{as:MNE}. Intuitively, Assumption \ref{as:MNE} is satisfied for  a large exponent $\beta$
if $\P_X$ has only a low concentration in the vicinity of the decision boundary, or if
$\P$ is particularly noisy in this region. For example, in the extreme case, in which $X_{-1}$ and $X_1$ have positive distance, we may choose arbitrarily large $\beta$, see also 
the first part of Example \ref{exa:noise_exponents} below for a simple example
of distributions having a low but positive concentration around the decision boundary.

In the following we want to give an overview of some influencing properties of the distribution $\P$ on Assumption \ref{as:TNC} and \ref{as:MNE}. We first state a proposition, which is the content of \cite[Lemma 8.23]{SteinwartChristmannSVMs}, that gives an important sufficient condition for Assumption \ref{as:MNE} to be satisfied and simultaneously relates it to another common regularity assumption for nonparametric classification.
\begin{prop}\label{prop:MNE}
Assume there exist constants $c,\alpha>0$ such that $|2\eta(x)-1|\leq c\Delta^\alpha(x)$ for $\P_X$-almost all $x\in X$ and that Assumption \ref{as:TNC} is satisfied for constants $C_*$ and $q$. Then Assumption \ref{as:MNE} is satisfied for $\beta=\alpha(q+1)$ and some constant $C_{**}$ only depending on $c$ and $C_*$.
\end{prop}

For $\alpha\leq 1$ the assumption in the proposition above can be seen as a substantially weaker form of $\alpha$-H\"older regularity for $\eta$, since for some $x_0\in\{x:\eta(x)=1/2\}$ attaining minimum distance to $x\in X$ this condition can be rewritten as $|\eta(x)-\eta(x_0)|\leq c|x-x_0|^\alpha/2$. That is, the H\"older condition does not need to be satisfied for \textit{arbitrary} $x,x_0\in X$ but only where one of the points considered is in $\{\eta=1/2\}$.
The following examples give some intuition on Assumption \ref{as:TNC} and especially \ref{as:MNE}, which is harder to interpret, and explain how they relate to each other.
\begin{exa}\label{exa:noise_exponents}
In the following let $X=[-1,1]^2$.
\begin{enumerate}
\item Assume $\P_X$ restricted to $A:=[-1/2,1/2]\times[-1,1]$ is proportional to a measure with Lebesgue-density $|x_1|^\sigma\mathrm{d}\lambda(x_1,x_2)$ for some $\sigma>0$ and on $X\backslash A$ is proportional to the uniform distribution. Further assume $\eta$ is the sawtooth function
\begin{align*}
\eta(x_1,x_2)=\begin{cases}
2(1-x_1) &\text{ for } x_1\in(1/2,1],\\
2x_1 &\text{ for } x_1\in[-1/2,1/2],\\
2(-1-x_1) &\text{ for } x_1\in[-1,-1/2).
\end{cases}
\end{align*}
Then by the behavior of $\P_X$ and $\eta$ near $x_1=\pm 1$ the optimal exponent in Assumption \ref{as:TNC} is given by $q=1$. By the low concentration of $\P_X$ near the decision boundary Assumption \ref{as:MNE} is fulfilled for the exponent $\beta=\sigma+2$.
\item Assume that $\P_X$ is uniform on $\{x\in[-1,1]^2:|x_2|\leq |x_1|^\zeta\}$ for some $\zeta>0$ and that $2\eta(x)-1=x_1$, i.e.~the classes $X_1,X_{-1}$ only meet at the point $(0,0)$ and not along a one-dimensional curve. Then Assumption \ref{as:TNC} is fulfilled for $q=\zeta+1$ and Assumption \ref{as:MNE} for $\beta=\zeta+2$.
\end{enumerate}
The first example shows that Assumption \ref{as:TNC} describes the global amount of mass close to the critical level $\eta=1/2$, while Assumption \ref{as:MNE} can additionally benefit from low concentration of $\P_X$ in the vicinity of the decision boundary. The second example shows that Assumption \ref{as:MNE} can also benefit from geometrical assumptions on the decision boundary, respectively the decision classes, and especially that the exponents in Assumption \ref{as:TNC} and \ref{as:MNE} can be simultaneously large. Also note that the exponent from Assumption \ref{as:TNC} can be deteriorated by the behavior of $\P_X$ or $\eta$ far away from the decision boundary, as seen in the first example, while Assumption \ref{as:MNE} is robust to such perturbations. Further note, that also a combination of the effects demonstrated by the examples above is possible.
\end{exa}
Before presenting the main results of this section, we briefly address the issue of the non-convexity of $\Lclass$. By Zhang's inequality, cf.~\cite[Theorem 2.31]{SteinwartChristmannSVMs} we have
\begin{align*}
\Rk_{\Lclass,\P}(f)-\Rk_{\Lclass,\P}^*\leq \Rk_{\Lhinge,\P}(f)-\Rk_{\Lhinge,\P}^*
\end{align*}
for all $f:X\rightarrow \R$. That is, by deriving learning rates for the loss $\Lhinge$, we implicitly derive the same learning rates for the loss $\Lclass$.

Our main results of this section are organized as in the previous section: We first present a general oracle inequality and then derive learning rates by a suitable choice of hyperparameters. Finally, we show that a TV-SVM achieves the same rates adaptively.
\begin{thm}\label{thm:class_oracle_inequality}
Assume $\P$ satisfies \ref{ass:boxdim_X} with parameters $\Cdim,\varrho$ as well as Assumptions \ref{as:TNC} and \ref{as:MNE} with parameters $C_*,q$ and $C_{**},\beta$ respectively. Then for the SVM using the hinge loss $L=\Lhinge$ we have for all $\tau>0,n>1,\lambda\in(0,1)$ and $n^{-1/\varrho}\leq \gamma\leq 1$
\begin{align}\label{eqn:class_oracle_inequality}
\begin{split}
\Rk_{L,\P}(\wideparen{f}_{D,\lambda,\gamma})-\Rk_{L,\P}^*\leq& c_1\lambda\gamma^{-d}+c_2C_{**}\gamma^\beta+c_3K\lambda^{-1/\log n}\left( \frac{\gamma^{-\varrho}}{n}\right)^\frac{q+1}{q+2}\log^{d+1}n \\
&+3C_*^\frac{q}{q+2}\left(\frac{432\tau}{n} \right)^\frac{q+1}{q+2}+30\frac{\tau}{n}
\end{split}
\end{align}
with probability $\P^n$ not less than $1-3\e^{-\tau}$, where $c_1=3^{d+2}/\Gamma(d/2+1)$,
\begin{align*}
c_2=9\frac{2^{1-\beta/2}\Gamma\left( \frac{\beta+d}{2}\right)}{\Gamma(d/2)}, \quad c_3=\max\left\{ C_*^{q/(q+1)},4\right\}\max\left\{ \Cdim,2\right\}
\end{align*}
and $K$ is a constant independent of $\P,n,\lambda$ and $\gamma$.

If in addition $\supp\,\P_X$ is contained in a $\varrho$-dimensional, infinitely differentiable, compact manifold then the factor $\log^{d+1}n$ on the right hand side of (\ref{eqn:class_oracle_inequality}) can be substituted by $\log^{\varrho+1}n$, where the constant $c_3$ then additionally depends on the curvature of the manifold.
\end{thm}
Using the theorem above, we can easily derive classification rates by choosing the hyperparameters $\lambda$ and $\gamma$ appropriately.
\begin{cor}\label{cor:class_learning_rates}
Let the assumptions of Theorem \ref{thm:ls_oracle_inequality} be satisfied and set $\gamma_n=n^{-a}$ and $\lambda_n=n^{-b}$ with
\begin{align*}
a= \frac{q+1}{\beta(q+2)+\varrho(q+1)}  \quad \text{ and } \quad  b\geq\frac{(d+\beta)(q+1)}{\beta(q+2)+\varrho(q+1)}.
\end{align*}
Then there exists a constant $C>0$ only depending on $\Cdim,C_*$ and $C_{**}$ such that for all $n>1$ and $\tau\geq 1$ we have
\begin{align*}
\Rk_{L,\P}(\wideparen{f}_{D,\lambda_n,\gamma_n})-\Rk_{L,\P}^*\leq C\tau\, n^{-\frac{\beta(q+1)}{\beta(q+2)+\varrho(q+1)}}\log^{d+1}n
\end{align*}
with probability $\P^n$ not less than $1-\e^{-\tau}$.

If in addition $\supp\,\P_X$ is contained in a $\varrho$-dimensional, infinitely differentiable, compact manifold then the factor $\log^{d+1}n$ above can be substituted by $\log^{\varrho+1}n$, where the constant $C$ then additionally depends on the curvature of the manifold.
\end{cor}
If $\eta$ satisfies the condition in Proposition \ref{prop:MNE} for some $\alpha>0$ the exponent in the rate in the corollary above is given by $\alpha(q+1)/(\alpha(q+2)+\varrho)$. For $\alpha\leq 1$ we see, using the short remark after Proposition \ref{prop:MNE}, that by \cite[Theorem 4.1]{AudibertTsybakov_FastLearningRatesForPlugInClassifiers} the rate we get from the corollary above is optimal up to a logarithmic factor and that we achieve this optimal rate for a substantially larger class of distributions than the class, for which the exact optimal rate was established in \cite{AudibertTsybakov_FastLearningRatesForPlugInClassifiers}. The statements of Remark \ref{rmk:optimality_reg_rates} on the optimality for $\varrho<d$ also apply in this case. The following theorem shows, that the rates of the corollary above can also be achieved adaptively without knowledge on $\P$.
\begin{thm}\label{thm:class_adaptive_rates}
Let $A_n$ be a minimal $1/\log n$-net of $(0,1]$ with $1\in A_n$ and let $B_n$ be a minimal $1/\log n$-net of $(0,d]$ with $d\in B_n$. Set $\Gamma_n:=\{n^{-a}:a\in A_n\}$ and $\Lambda_n:=\{n^{-b}:b\in B_n\}$. Let the assumptions of Theorem \ref{thm:class_oracle_inequality} be satisfied and assume $\varrho\geq 1$. Then there exists a constant $C>0$ only depending on $\Cdim,C_*$ and $C_{**}$ such that the TV-SVM using $\Lambda_n$ and $\Gamma_n$ satisfies for all $n>1$ and $\tau\geq 1$,
\begin{align*}
\Rk_{L,\P}(\wideparen{f}_{D_1,\lambda_{D_2},\gamma_{D_2}})-\Rk_{L,\P}^*\leq C\tau\, n^{-\frac{\beta(q+1)}{\beta(q+2)+\varrho(q+1)}}\log^{d+1}n
\end{align*}
with probability $\P^n$ not less than $1-\e^{-\tau}$.

If in addition $\supp\,\P_X$ is contained in a $\varrho$-dimensional, infinitely differentiable, compact manifold then the factor $\log^{d+1}n$ above can be substituted by $\log^{\varrho+1}n$, where the constant $C$ then additionally depends on the curvature of the manifold.
\end{thm}
The statements of Remark \ref{rmk:validation_set} on the size of the candidate set $\Lambda_n$ and the need for the constraint $\varrho\geq 1$ apply one-to-one for Theorem \ref{thm:class_adaptive_rates}.

We again conclude the section with a comparison to existing results on classification rates under a low intrinsic dimensionality assumption similar to ours, although they are much rarer than for regression.
\begin{description}
\item[SVMs.] In \cite{YeZhou_SVMLearningAndLpApproximationByGaussianOnRiemannianManifolds} the authors assume that the input space $X\subset\R^d$ is a compact, connected, smooth $\varrho$-dimensional manifold without boundary and $\P_X$ is the normalized surface measure on $X$. They further assume, that $\mathrm{sgn}(2\eta-1)$ is contained in the interpolation space $(L_1(X),W^{2,1}(X))_{\theta,\infty}$ for some $\theta\in(0,1]$, where $W^{2,1}(X)$ is a Sobolev-space on the manifold $X$ and derive the learning rate $(\log^2(n)/n)^{\theta/(6\theta+\varrho)}$ for Gaussian SVMs. Of course, such a regularity assumption for a discrete-valued function is very restrictive, especially for small values of $\varrho$. Exemplarily, for $\varrho=1$ and $\theta\geq 1/2$ their assumptions actually imply the pathological case $\mathrm{sgn}(2\eta-1)\equiv 1$ or $\mathrm{sgn}(2\eta-1)\equiv -1$, since by the embedding theorem in \cite[7.4.2 (iv)]{Triebel_TheoryOfFunctionSpacesII} the space $(L_1(X),W^{2,1}(X))_{\theta,\infty}$ then consists of continuous functions. In addition, their fastest possible rate is given by $n^{-1/7}$, while we derive learning rates up to $n^{-1}$.
\item[Dyadic Decision Trees.] In \cite{ScottNowak_MinimaxOptimalClassificationWithDyadicDecisionTrees} dyadic decision trees are considered. They assume, that for a partition $\mathcal{P}_m$ of the input space $X=[0,1]^d$ into cubes of sidelength $1/m$, where $m$ is a dyadic integer, every $A\in \mathcal{P}_m$ satisfies $\P_X(A)\leq c_0 m^{-\varrho}$. Additionally they assume, that the number of cubes in $\mathcal{P}_m$, that intersect the decision boundary $\{x\in X:\eta(x)=1/2\}$ is bounded by $c_1 m^{\varrho -1}$, that is they impose an assumption similar to \ref{ass:boxdim_X} on the decision boundary. They derive a learning rate of $(\log(n)/n)^{1/\varrho}$. Remarkably, the optimal choice of their only hyperparameter, the depth of the tree, does not depend on $\varrho$.
\end{description}

\begin{appendix}
\section{Proofs and Auxiliary Results}\label{sec:appendix}
In the following we present the proofs of the statements in the main sections of this work. We first state and prove some auxiliary results on entropy numbers for Gaussian RKHSs in Section \ref{sec:entropy_estimates}. Based on these results we derive an oracle inequality for Gaussian SVMs for generic loss functions under assumption \ref{ass:boxdim_X} in Section \ref{sec:general_oracle_inequality}. Sections \ref{sec:proofs_regression} and \ref{sec:proofs_class} contain the proofs of Sections \ref{sec:learning_rates_regression} and \ref{sec:learning_rates_classification}, respectively.
\subsection{Entropy Estimates}\label{sec:entropy_estimates}
Our main tool for bounding the statistical error is based on the entropy numbers of $H_\gamma(X)$ with respect to the empirical $L_2(\D)$-norm, where $\D:=n^{-1}\sum_{i=1}^n \delta_{x_i}$ is the empirical measure associated to a sample $D=(x_1,\ldots,x_n)\in X^n$ drawn from $\P_X^n$. To this end, we first give a definition of entropy numbers, which are the inverse concept of covering numbers.
\begin{dfn}\label{dfn:entropy_numbers}
Given normed spaces $E,F$ and a bounded, linear operator $T:E\rightarrow F$, for $i\in\N$ the $i$-th dyadic entropy number of $T$ is defined as
\begin{align*}
e_i(T):=\inf \left\{ \varepsilon>0:\exists x_1,\ldots,x_{2^{i-1}}\in F \text{ such that } TB_E\subset\bigcup_{j=1}^{2^{i-1}}(x_j+\varepsilon B_F)\right\}.
\end{align*}
\end{dfn}
As one would expect, there is a close connection between entropy numbers and covering numbers. For example we have $e_i(T)\in\O(i^{-1/q})$ as $i\rightarrow\infty$ for some $q>0$ if and only if $\log\cN(T,\varepsilon)\in\O(\varepsilon^q)$ as $\varepsilon\rightarrow0$, cf. \cite[Lemma 6.21 and Exercise 6.8]{SteinwartChristmannSVMs}.\\
We are interested in bounding the averaged entropy numbers of the embedding $\id:H_\gamma(X)\rightarrow L_2(\D)$. More precisely, we need a polynomial bound of the form
\begin{align}
\E_{D\sim \P_X^n} e_i(H_\gamma(X)\rightarrow L_2(\D)) \leq a\, i^{-\frac{1}{2p}} \quad\text{ for all }\quad i\in\N \label{eqn:bound_average_entropy_numbers}
\end{align}
for some constants $a>0$ and $p\in(0,1)$. A bound of the form (\ref{eqn:bound_average_entropy_numbers}) in turn is a tool for bounding Rademacher averages of $H_\gamma(X)$ using a standard symmetrization procedure and Dudley's chaining, see \cite[Section 7.3]{SteinwartChristmannSVMs}. We will derive (\ref{eqn:bound_average_entropy_numbers}) from $\ell_\infty(X)$-covering numbers of the unit ball in $H_\gamma(X)$, where $\ell_\infty(X)$ denotes the space of bounded functions on $X$ equipped with the sup-norm $\norm{\cdot}_\infty$. Covering numbers for Gaussian RKHSs with respect to $\norm{\cdot}_\infty$ are well understood, for example \cite[Theorem 3]{KuehnEntrop} showed that
\begin{align}
\log \cN(\id:H_\gamma([0,1]^d)\rightarrow\ell_\infty([0,1]^d),\varepsilon)\asymp \dfrac{\left( \log \frac{1}{\varepsilon} \right)^{d+1}}{\left(\log\log \frac{1}{\varepsilon}\right)^d} \quad \text{ as } \quad \varepsilon \rightarrow 0.\label{eqn:Kuehn_covering_numbers}
\end{align}
However, we will rely on the slightly suboptimal bound
\begin{align}
\log \cN(\id:H_\gamma([0,1]^d)\rightarrow\ell_\infty([0,1]^d),\varepsilon)\lesssim \log^{d+1} \frac{1}{\varepsilon}\quad \text{ as } \quad \varepsilon \rightarrow 0,\label{eqn:Vaart_covering_numbers}
\end{align}
since it is very hard to make use of the extra double logarithmic factor in (\ref{eqn:Kuehn_covering_numbers}). One can show that (\ref{eqn:Vaart_covering_numbers}) actually implies super-polynomial decay of $\E_{D\sim\P_X^n} e_i(\id:H_\gamma(X)\rightarrow L_2(\D))$, meaning (\ref{eqn:bound_average_entropy_numbers}) is fulfilled for arbitrarily small $p>0$ with a constant $a$ necessarily depending on $p$. For our statistical analysis we need to explicitly track the dependence of $a$ on $p$ and $\gamma$. This is done in the following theorem and its corollary, whose proofs can be found later in this subsection after we have established some auxiliary lemmas.
\begin{thm}\label{thm:l_infty_entropy_numbers}
There exists a universal constant $K$ only depending on $d$, such that
\begin{align*}
e_i(\id:H_\gamma(X)\rightarrow \ell_\infty(X))\leq K^{\frac{1}{2p}}p^{-\frac{d+1}{2p}}\cN_{\ell_\infty^d}(X,\gamma)^{\frac{1}{2p}}\,i^{-\frac{1}{2p}}
\end{align*}
holds for all $i\in\N$, $p\in(0,1)$ and $\gamma>0$.
\end{thm}
Combining Theorem \ref{thm:l_infty_entropy_numbers} with assumption \ref{ass:boxdim_X} easily gives a bound of the form (\ref{eqn:bound_average_entropy_numbers}) as desired.
\begin{cor}\label{cor:averaged_entropy_numbers}
Let $\P$ satisfy \ref{ass:boxdim_X}. Then there exists a constant $K$ only depending on $d$, such that the bound
\begin{align*}
\E_{D\sim\P_X^n}e_i( \id:H_\gamma(X)\rightarrow L_2(\D)) \leq (\Cdim K)^{\frac{1}{2p}}p^{-\frac{d+1}{2p}}\gamma^{-\frac{\varrho}{2p}}\,i^{-\frac{1}{2p}}
\end{align*}
holds for all $i\in\N$, $p\in(0,1)$ and $\gamma\in(0,1)$.
\end{cor}
Note, that the behavior of the bandwidth $\gamma$ in the entropy estimate above now depends on $\varrho$ instead of $d$, compared to e.g.~\cite[Lemma 4.5]{VaartZantenEntrop}, which is a crucial improvement.

As already mentioned, the proofs of Theorem \ref{thm:l_infty_entropy_numbers} and Corollary \ref{cor:averaged_entropy_numbers} need some preparatory lemmas on kernels and their reproducing kernel Hilbert spaces. To this end, we recall some basic definitions and facts, which can be found in \cite[Section 4]{SteinwartChristmannSVMs}. Given a non-empty set $X$, a symmetric function $k:X\times X\rightarrow\R$ is a \textit{kernel}, if it is positive definite, that is if for all $n\in\N$ and all choices $x_1,\ldots,x_n\in X$ and $\alpha_1,\ldots,\alpha_n\in\R$ we have
\begin{align*}
\sum_{i=1}^n \sum_{j=1}^n \alpha_i \alpha_j k(x_j,x_i)\geq 0.
\end{align*}
A Hilbert space $H$ of functions $f:X\rightarrow\R$ is called a \textit{reproducing kernel Hilbert space} (RKHS), if the evaluation functional $H\rightarrow\R$ defined by $f\mapsto f(x)$ is continuous for every $x\in X$. Every RKHS $H$ has a unique \textit{reproducing kernel} $k: X\times X\rightarrow \R$, i.e.~$k(\cdot,x)\in H$ for all $x\in X$ and
\begin{align}
\langle f,k(\cdot,x)\rangle_H=f(x) \quad \text{for all}\quad f\in H, x\in X, \label{eqn:reproducing_property}
\end{align}
and $k$ is a kernel in the sense above. Property (\ref{eqn:reproducing_property}) is called the \textit{reproducing property}. Conversely, for every  kernel $k$ there exists a unique RKHS $H$ for which it is the reproducing kernel. An equivalent definition is that $k:X\times X\to\R$ is a kernel, if there exists a Hilbert space $H_0$ and a map $\Phi:X\to H_0$ such that $k(x,y)=\langle \Phi(x),\Phi(y) \rangle_{H_0}$ for all $x,y\in X$. In this context $H_0$ is called a \textit{feature space} and $\Phi$ a \textit{feature map}. Note that we can always pick the RKHS $H$ as a feature space via the \textit{canonical feature map} $\Phi:X\rightarrow H$ defined by $\Phi(x):=k(x,\cdot)$ and that the inner product on $H$ is entirely determined by the kernel.
\begin{lem}\label{lem:transformed_kernel}
Let $k$ be a kernel on $X$ with RKHS $H$ and let $\psi:Y\rightarrow X$ be a map. Then $k_\psi(\cdot,\cdot):=k(\psi(\cdot),\psi(\cdot))$ is a kernel on $Y$ with RKHS $H_\psi=\{f\circ\psi:f\in H\}$ and the map $V:H\rightarrow H_\psi$ defined by $ f\mapsto f\circ\psi$ is a metric surjection. The norm in $H_\psi$ can be computed by
\begin{align*}
\norm{g}_{H_\psi}=\inf\{\norm{f}_H: f \text{ with } g=f\circ \psi \}.
\end{align*}
If $\psi$ is bijective, then $V$ is an isometric isomorphism.
\end{lem}
\begin{proof}
Let $\Phi:X\rightarrow H, x\mapsto k(x,\cdot)$ be the canonical feature map of $k$ and define $\Phi_\psi:Y\rightarrow H, y\mapsto \Phi(\psi(y))$. Then by construction we have $\langle \Phi_\psi(y),\Phi_\psi(y') \rangle_H=k_\psi(y,y')$ for all $y,y'\in Y$, that is, $\Phi_\psi$ is a feature map of $k_\psi$. The first two assertions now follow from \cite[Theorem 4.21]{SteinwartChristmannSVMs}. For the third assertion additionally apply this result on $\psi^{-1}$.
\end{proof}
\begin{cor}\label{cor:restricted_kernel}
Let $k$ be a kernel on $X\subset\R^d$, $H$ its RKHS and $Y\subset X$. Then $H|_Y:=\{f|_Y:f\in H\}$ is the RKHS of $k|_{Y\times Y}$ and the restriction $H\rightarrow H|_Y$ is a metric surjection.
\end{cor}
\begin{proof}
This follows from Lemma \ref{lem:transformed_kernel} with $\psi:Y\rightarrow X$ being the inclusion.
\end{proof}
A kernel $k:\R^d\times\R^d \rightarrow\R$ is called radial if there exists a function $\kappa:[0,\infty)\rightarrow\R$ such that $k(x,y)=\kappa(\norm{x-y})$ for all $x,y\in\R^d$. Radial kernels are a special case of translation invariant kernels, which by Bochner's theorem \cite[Theorem IX.9]{ReedSimon_MethodsOfModernMathematicalPhysicsII} are characterized as the inverse Fourier transform of a finite Borel measure on $\R^d$. Similarly, by Schoenberg's theorem \cite{Schoenberg_MetricSpacesAndCompletelyMonotoneFunctions} $\kappa:[0,\infty)\rightarrow\R$ defines a radial kernel on $\R^d$ via $k(\cdot,\cdot)=\kappa(\norm{\cdot-\cdot})$ for every $d\in\N$, if and only if there exists a finite Borel measure $\mu$ on $[0,\infty)$ such that
\begin{align*}
\kappa(t)=\int_0^\infty \e^{-xt^2}\d\mu(x)
\end{align*}
for all $t\in[0,\infty)$, i.e.~radial kernels are mixtures of Gaussians. The following corollary shows, that RKHSs of radial kernels are in some sense translational and rotational invariant.
\begin{cor}\label{cor:shifted_kernel}
Let $k$ be a radial kernel on $\R^d$ and for $X\subset \R^d$ denote the restriction of $k$ onto $X\times X$ by $k_X$  and its RKHS by $H(X)$. Fix an $a\in\R^d$ and an orthogonal matrix $U\in \R^{d\times d}$. Then the operator $T:H(X)\rightarrow H(a+UX)$ defined by $ Tf(x)=f(U^{-1}(x-a))$ is well-defined and an isometric isomorphism.
\end{cor}
\begin{proof}
This follows from Lemma \ref{lem:transformed_kernel} with the map $\psi: a+UX\rightarrow X$ defined by $ x\mapsto U^{-1}(x-a)$, since $k_X(\psi(\cdot),\psi(\cdot))=k_{a+UX}(\cdot,\cdot)$.
\end{proof}
\begin{lem}\label{lem:linfty_covering_numbers_partition}
Let $k$ be a kernel on $X$, $H$ be its RKHS, and $X_1,\ldots,X_N\subset X$ pairwise disjoint subsets with $X_1\cup\ldots\cup X_N=X$. Then for all $\varepsilon>0$ we have
\begin{align*}
\cN_{\ell_\infty(X)}(B_H,\varepsilon)\leq \prod_{k=1}^N \cN_{\ell_\infty(X_k)}(B_{H|_{X_k}},\varepsilon).
\end{align*}
\end{lem}
\begin{proof}
As the general statement easily follows inductively, we will only prove the case $N=2$. To this end, let $f_1,\ldots,f_n\in\ell_\infty(X_1)$ be a minimal $\varepsilon$-net of $B_{H|_{X_1}}$ and let $g_1,\ldots,g_m\in\ell_\infty(X_2)$ be  a minimal $\varepsilon$-net of $B_{H|_{X_2}}$. Let $f\in B_H$. Then by Corollary \ref{cor:restricted_kernel} we have $f|_{X_l}\in H|_{X_l}$ with $\norm{f|_{X_l}}_{H|_{X_l}}\leq\norm{f}_H\leq 1$ for $l=1,2$. Hence, there exist $i,j$ with $\norm{f|_{X_1}-f_i}_{\ell_\infty(X_1)}\leq \varepsilon$ and $\norm{f|_{X_2}-g_j}_{\ell_\infty(X_2)}\leq \varepsilon$. If we denote the zero-extensions of $f_i,g_j$ to $X$ by $\widehat{f}_i, \widehat{g}_j$, we see that $\{\widehat{f}_i+\widehat{g}_j:i=1,\ldots,n,\, j=1,\ldots,m \}$ is a $\varepsilon$-net of $B_H$ with cardinality $n\cdot m$.
\end{proof}
\begin{cor}\label{cor:aux_entropy_bound}
For all $\gamma>0$ and $\varepsilon>0$ we have
\begin{align*}
&\log\cN(\id:H_\gamma(X)\rightarrow\ell_\infty(X),\varepsilon) \\
\leq & \, \cN_{\ell_\infty^d}(X,\gamma)\cdot\log\cN\left(\id:H_1([-1,1]^d)\rightarrow\ell_\infty([-1,1]^d),\varepsilon\right).
\end{align*}
\end{cor}
\begin{proof}
Let $x_1,\ldots,x_n\in\R^d$ be a minimal $\gamma$-net of $X$ w.r.t.~$\ell_\infty^d$. We partition $X$ into $X_1,\ldots,X_n$, where $X_j$ consists of the points $x\in X$ that are $\ell_\infty^d$-closest to $x_j$. Here we break ties, for example, in favor of a smaller index $j$. Combining Lemma \ref{lem:linfty_covering_numbers_partition}, Corollary \ref{cor:restricted_kernel} and Corollary \ref{cor:shifted_kernel} we get
\begin{align*}
\log\cN_{\ell_\infty(X)}\left(B_{H_\gamma(X)},\varepsilon\right) &\leq  \sum_{j=1}^n \log\cN_{\ell_\infty(X_j)}\left(B_{H_\gamma(X_j)},\varepsilon\right) \\
&\leq  \sum_{j=1}^n \log\cN_{\ell_\infty(x_j+\gamma [-1,1]^d)}\left(B_{H_\gamma(x_j+\gamma[-1,1]^d)},\varepsilon\right) \\
&= n\log\cN_{\ell_\infty(\gamma[-1,1]^d)}\left(B_{H_\gamma(\gamma[-1,1]^d)},\varepsilon\right).
\end{align*}
The result now follows from \cite[Proposition 4.37]{SteinwartChristmannSVMs}, which states that the scaling operator $\tau_\gamma:H_\gamma(\gamma[-1,1]^d)\rightarrow H_1([-1,1]^d)$ defined by $ \tau_\gamma f(x)=f(\gamma x)$ is an isometric isomorphism.
\end{proof}
\begin{proof}[Proof of Theorem \ref{thm:l_infty_entropy_numbers}]
For $f\in B_{H_1([-1,1]^d)}$ by \cite[Lemma 4.23]{SteinwartChristmannSVMs} we have $\norm{f}_\infty\leq 1$
and consequently we find $\cN(B_{H_1([-1,1]^d)},\varepsilon)=1$ for all $\varepsilon\geq 1$. Furthermore, by \cite[Theorem 3]{KuehnEntrop} there exists a constant $K\geq 1$ such that 
\begin{align*}
\log\cN\left( B_{H_1([-1,1]^d)},\varepsilon \right)\leq K \log^{d+1}\frac{2}{\varepsilon}
\end{align*}
for all $\varepsilon\in(0,1]$. Some elementary calculations show that
\begin{align*}
\sup_{\varepsilon\in(0,2)} \varepsilon^q \log^{d+1}\frac{2}{\varepsilon} = 2^q \left( \frac{d+1}{\e q} \right)^{d+1},
\end{align*}
which combined with Corollary \ref{cor:aux_entropy_bound} and the estimate above gives us
\begin{align}\label{eqn:log_cover_gauss}
\begin{split}
\log\cN(\id:H_\gamma(X)\rightarrow\ell_\infty(X),\varepsilon) &\leq K\cN_{\ell_\infty^d}(X,\gamma) \log^{d+1}\frac{2}{\varepsilon} \\
&\leq  4K\cN_{\ell_\infty^d}(X,\gamma) \left( \frac{d+1}{\e q} \right)^{d+1}\varepsilon^{-q}
\end{split}
\end{align}
for $\varepsilon>0$ and $q\in(0,2)$. As a final step we convert the latter bound on the covering numbers of $ \id:H_\gamma(X)\rightarrow\ell_\infty(X) $ into a bound on the entropy numbers. To this end, we fix an $i\geq2$ and define $\varepsilon>0$ by $\exp (a/\varepsilon^q)=2^{i-1} $, where
\begin{align*}
a:=4K\cN_{\ell_\infty^d}(X,\gamma) \left( \frac{d+1}{\e q} \right)^{d+1}.
\end{align*}
By (\ref{eqn:log_cover_gauss}) this implies
\begin{align*}
e_i(\id:H_\gamma(X)\rightarrow\ell_\infty(X))\leq \left( \frac{(i-1)\log 2}{a} \right)^{-\frac{1}{q}} \leq \left(\frac{2a}{\log 2} \right)^\frac{1}{q} i^{-\frac{1}{q}}
\end{align*}
for all $i\geq 2$. Since $e_1(\id:H_\gamma(X)\rightarrow\ell_\infty(X))\leq 1$ we get
\begin{align*}
e_i(\id:H_\gamma(X)\rightarrow\ell_\infty(X))\leq \left( \frac{\cN_{\ell_\infty^d}(X,\gamma) 8K}{\log 2} \left( \frac{d+1}{\e q} \right)^{d+1} \right)^\frac{1}{q} i^{-\frac{1}{q}}
\end{align*}
for all $i\in\N$. Now substitute $2p=q$ and absorb all irrelevant constants into a new constant $K$.
\end{proof}
For the proof of Corollary \ref{cor:averaged_entropy_numbers} we will use the basic fact that entropy numbers are dominated by decompositions, that is, if we can decompose $T:E\rightarrow F$ into $ T=RS$ with bounded, linear operators $S:E\rightarrow \tilde{F}$, $R:\tilde{F}\rightarrow F$, and an intermediate normed space $\tilde{F}$, then we have $e_i(T)\leq \norm{R} e_i(S)$ as well as $e_i(T)\leq e_i(R)\norm{S}$ for all $i\in\N$, c.f.~\cite[Section 1.3]{CarlStephani_EntropyCompactness}, where $\norm{R}$, $\norm{S}$ denotes the operator norm.
\begin{proof}[Proof of Corollary \ref{cor:averaged_entropy_numbers}]
Consider the decomposition of $\id:H_\gamma(X)\rightarrow L_2(\D)$ for a sample $D\in(\supp\,\P_X)^n$ described by the commutative diagram
\begin{center}
\begin{tikzcd}[row sep=2cm, column sep = 3cm]
H_\gamma(X) \arrow[r, "\mathrm{id}"] \arrow[d, "\mathrm{res}", swap]
& L_2(\D)  \\
H_\gamma(\supp\,\P_X) \arrow[r, "\mathrm{id}", swap]
& \ell_\infty(\supp\,\P_X) \arrow[u, "\mathrm{id}", swap]
\end{tikzcd}
\end{center}
where $\mathrm{res}:H_\gamma(X)\rightarrow H_\gamma(\supp\,\P_X)$ is the restriction operator. We have $\norm{\mathrm{res}:H_\gamma(X)\rightarrow H_\gamma(\supp,\P_X)}\leq 1$ by Corollary \ref{cor:restricted_kernel} and trivially also $\norm{\id:\ell_\infty(\supp\,\P_X)\rightarrow L_2(\D)}\leq 1$. Theorem \ref{thm:l_infty_entropy_numbers} then implies
\begin{align*}
e_i(\id:H_\gamma(X)\rightarrow L_2(\D))\leq K^\frac{1}{2p} p^{-\frac{d+1}{2p}} \cN_{\ell_\infty^d}(\supp\,\P_X,\gamma)^\frac{1}{2p}\, i^{-\frac{1}{2p}}
\end{align*}
for all $i\in \N$, $\gamma>0$ and $p\in (0,1)$ for $\P_X^n$-almost all $D\in(X\times Y)^n$. Now combine the above bound with assumption \ref{ass:boxdim_X} and take the expectation w.r.t.~$D\sim \P_X^n$.
\end{proof}

\subsection{A General Oracle Inequality}\label{sec:general_oracle_inequality}
In this section we derive a general oracle inequality for bounding the excess risk of our estimator $f_{D,\lambda,\gamma}$. A main ingredient for this oracle inequality is a so-called variance bound, which has been proven to be a useful tool for deriving fast learning rates, see for example \cite{Bousquet_NewApproachesToStatisticalLearningTheory}, especially the discussion in Section 5.2 therein.
\begin{dfn}\label{dfn:variance_bound}
Let $L$ be a loss that can be clipped at $M>0$ and let $\mathcal{F}$ be some function class of measurable functions $f:X\rightarrow\R$. Assume there exists a Bayes decision function $f_{L,\P}^*:X\rightarrow[-M,M]$. We say, that the supremum bound is satisfied, if there exists a constant $B>0$, such that $L(y,t)\leq B$ for all $(y,t)\in Y\times[-M,M]$. We further say, that the variance bound is satisfied, if there exist $\vartheta\in[0,1]$ and $V\geq B^{2-\vartheta}$, such that
\begin{align*}
\E (L\circ \wideparen{f}-L\circ f_{L,\P}^*)^2 \leq V\cdot \left( \E \,L\circ \wideparen{f}-L\circ f_{L,\P}^*\right)^\vartheta
\end{align*}
for all $f\in\mathcal{F}$.
\end{dfn}
Intuitively, the variance bound ensures a small variance of the excess risk, whenever our estimator is close to the optimum.

The main results of Sections \ref{sec:learning_rates_regression} and \ref{sec:learning_rates_classification} are based on the following theorem, which may also serve as a blueprint for deriving learning rates for Gaussian SVMs under assumption \ref{ass:boxdim_X} using other loss functions than the ones considered in this paper. To this end we first introduce some minimum regularity assumption, that we need to impose on the loss function $L$. We say that a loss function $L:Y\times\R\rightarrow [0,\infty)$ is locally Lipschitz continuous if for every $a>0$ the functions $L(y,\cdot)|_{[-a,a]}, y\in Y$ are uniformly Lipschitz continuous, that is
\begin{align*}
|L|_{a,1}:=\sup_{\substack{s,t\in [-a,a], s\neq t \\y\in Y} } \frac{|L(y,t)-L(y,s)|}{|t-s|}<\infty.
\end{align*}
The following theorem establishes an oracle inequality for Gaussian SVMs under assumption \ref{ass:boxdim_X} for generic loss functions.
\begin{thm}\label{thm:general_oracle_inequality}
Assume $L$ is a locally Lipschitz continuous loss that can be clipped at $M>0$ and that the supremum and variance bounds are satisfied for constants $B>0, \vartheta\in[0,1]$, and $V\geq B^{2-\vartheta}$. Furthermore, assume $X$ satisfies \ref{ass:boxdim_X} for $\Cdim>0$ and $0<\varrho\leq d$ and fix an $f_0\in H_\gamma(X)$ and a $B_0\geq B$ with $\norm{L\circ f_0}_\infty\leq B_0$. Then there exists a constant $K$ such that for all $n\in\N, \gamma\in(0,1),\lambda>0, p\in (0,1/2]$ and $\tau>0$ we have
\begin{align}\label{eqn:general_oracle_inequality}
\begin{split}
\Rk_{L,\P}(\wideparen{f}_{D,\lambda,\gamma})-\Rk_{L,\P}^*\leq &\,9(\lambda\norm{f_0}_{H_\gamma(X)}^2+\Rk_{L,\P}(f_0)-\Rk_{L,\P}^*) \\
&+ C_\P K \left( \dfrac{p^{-d-1}\gamma^{-\varrho}}{\lambda^p n} \right)^\frac{1}{2-p-\vartheta+\vartheta p}+3\left( \frac{72V\tau}{n} \right)^\frac{1}{2-\vartheta}+\frac{15B_0 \tau}{n}
\end{split}
\end{align}
with probability $\P^n$ not less than $1-3\e^{-\tau}$, where $K$ is independent of $\P$ and
\begin{align*}
C_\P=\max\left\{B, \left( |L|_{M,1}^pV^\frac{1-p}{2}\right)^\frac{2}{2-p-\vartheta+\vartheta p}, |L|_{M,1}^pB^{1-p},1 \right\}\cdot \max\left\{\Cdim,B^{2p} \right\}^\frac{1}{2-p-\vartheta+\vartheta p}.
\end{align*}
If $\supp\,\P_X$ is contained in a $\varrho$-dimensional, infinitely differentiable, compact manifold then the factor $p^{-d-1}$ on the right hand side of (\ref{eqn:general_oracle_inequality}) can be substituted by $p^{-\varrho-1}$, where the constant $C_\P$ then additionally depends on the curvature of the manifold.
\end{thm}
\begin{proof}
By \cite[Theorem 7.23]{SteinwartChristmannSVMs} in combination with the entropy estimate from Corollary \ref{cor:averaged_entropy_numbers} we have
\begin{align*}
\Rk_{L,\P}(\wideparen{f}_{D,\lambda,\gamma})-\Rk_{L,\P}^*\leq & \,9(\lambda\norm{f_0}_{H_\gamma(X)}^2+\Rk_{L,\P}(f_0)-\Rk_{L,\P}^*) \\
&+ K(p) \left( \dfrac{a^{2p}}{\lambda^p n} \right)^\frac{1}{2-p-\vartheta+\vartheta p}+3\left( \frac{72V\tau}{n} \right)^\frac{1}{2-\vartheta}+\frac{15B_0 \tau}{n}
\end{align*}
with probability not less than $1-3\e^{-\tau}$, where
\begin{align*}
a=\max\left\{(\Cdim K)^\frac{1}{2p}p^{-\frac{d+1}{2p}}\gamma^{-\frac{\varrho}{2p}},B\right\}\leq p^{-\frac{d+1}{2p}}\gamma^{-\frac{\varrho}{2p}}\max\left\{ (\Cdim )^\frac{1}{2p},B \right\}K^\frac{1}{2p}
\end{align*}
with the constant $K$ from Corollary \ref{cor:averaged_entropy_numbers} and $K(p)$ is given by (see \cite[proof of Theorem 7.23]{SteinwartChristmannSVMs})
\begin{align*}
K(p)=&\max \big\{ 2B, 3\left( 30\cdot 2^p C_1(p) |L|_{M,1}^p V^\frac{1-p}{2} \right)^\frac{2}{2-p-\vartheta+\vartheta p},
90\cdot 120^p C_2^{1+p}(p) |L|_{M,1}^p B^{1-p} \big\} \\
\leq& \max\left\{ B, \left(|L|_{M,1}^p V^\frac{1-p}{2} \right)^\frac{2}{2-p-\vartheta+\vartheta p},|L|_{M,1}^pB^{1-p},1  \right\} \\
&\cdot\max\left\{2,3\left(30\cdot 2^pC_1(p) \right)^\frac{2}{2-p-\vartheta+\vartheta p},90\cdot 120^pC_2^{1+p}(p) \right\}.
\end{align*}
The constants $C_1(p)$ and $C_2(p)$ are given by
\begin{align*}
C_1(p):=\dfrac{2\sqrt{\log 256}C_p^p}{(\sqrt{2}-1)(1-p)2^\frac{p}{2}},\quad C_2(p)=\left(\dfrac{8\sqrt{\log 16} C_p^p}{(\sqrt{2}-1)(1-p)4^p}\right)^\frac{2}{1+p},
\end{align*}
where
\begin{align*}
C_p=\dfrac{\sqrt{2}-1}{\sqrt{2}-2^\frac{2p-1}{2p}}\cdot\dfrac{1-p}{p},
\end{align*}
which can be tracked in \cite[proof of Theorem 7.16]{SteinwartChristmannSVMs}. It was shown in \cite[Proof of Theorem 7]{FarooqSteinwart_LearningRatesforKernelBasedExpectileRegression} that $C_1(p)$ and $C_2^{1+p}(p)$ are uniformly bounded in $p\in(0,\frac{1}{2}]$, more precisely, we have
\begin{align*}
\sup_{p\in(0,\frac{1}{2}]}C_1(p)\leq 46\e \quad \text{and} \quad \sup_{p\in(0,\frac{1}{2}]}C_2^{1+p}(p)\leq 1035\e^2,
\end{align*}
which implies
\begin{align*}
K(p)\leq \tilde{K}\max\left\{ B, \left(|L|_{M,1}^p V^\frac{1-p}{2} \right)^\frac{2}{2-p-\vartheta+\vartheta p},|L|_{M,1}^pB^{1-p},1  \right\}
\end{align*}
with a constant $\tilde{K}$ independent of $p$ and $\vartheta$. To prove the statement under the more restrictive manifold assumption we apply \cite[Theorem 7.23]{SteinwartChristmannSVMs} with a constant $a\lesssim p^{-(\varrho+1)/(2p)}\gamma^{-\varrho/(2p)}$ by using the covering number bound from \cite[Lemma 5.4]{YangDunson_BayesianManifoldRegression}. The rest of the proof remains unchanged.
\end{proof}

\subsection{Proofs related to Section \ref{sec:learning_rates_regression} }\label{sec:proofs_regression}
\begin{lem}\label{lem:meanvaluetheorem}
Let $f:(a,b)\rightarrow\R$ be $s$-times continuously differentiable. Furthermore, fix $x\in (a,b)$, $h>0$, and $k\in\{1,\ldots,s\}$ with $(x,x+sh)\subset(a,b)$. Then there exists a $\xi\in (x,x+kh)$ such that $h^{-s}\Delta_h^sf(x)=h^{k-s}\Delta_h^{s-k}f^{(k)}(\xi)$, where $\Delta_h^s$ is the difference operator defined by (\ref{eqn:difference_operator}).
\end{lem}
\begin{proof}
Because of $\Delta_h^s=\Delta_h\Delta_h^{s-1}$ we can apply the mean value theorem to the function $h^{-1}\Delta_h^{s-1}f$, which gives us $h^{-s}\Delta_h^s f(x)=h^{1-s}\Delta_h^{s-1}f'(\xi)$ for some $\xi\in(x,x+h)$. Note that in the last step we used $\frac{\mathrm{d}}{\mathrm{d}x}\Delta_hf(x)=\Delta_hf'(x)$. That is, we have proven the assertion for $k=1$. Now we can iterate this argument by applying the mean value theorem to $h^{-1}\Delta_h^{s-2}f'$ and so on.
\end{proof}
\begin{proof}[Proof of Proposition \ref{prop:finite_Besov_norm}]
Fix an $x\in X$ and an $h\in\R^d$ with $\norm{h}<\delta/(s-1)$, where $s:=\lfloor\alpha\rfloor +1$, and define the univariate function $F(t):=f(x+th)$. Then we have $\Delta_h^sf(x)=\Delta_1^sF(0)=\Delta_1F^{(s-1)}(\xi)$ by Lemma \ref{lem:meanvaluetheorem} for some $\xi\in(0,s-1)$. The $(s-1)$-th derivative of $F$ is given by
\begin{align*}
\dfrac{\mathrm{d}^{s-1}}{\mathrm{d}t^{s-1}}F(t)=\sum_{|\nu|=s-1} \dfrac{(s-1)!}{\nu!} h^\nu \,\partial^\nu f(x+th).
\end{align*}
When $\alpha$ is not an integer this leads us to the estimate
\begin{align*}
|\Delta_1F^{(s-1)}(\xi)|&=|F^{(s-1)}(\xi+1)-F^{(s-1)}(\xi)| \\
&=  \left|\sum_{|\nu|=s-1} \dfrac{(s-1)!}{\nu!}h^\nu\big(\partial^\nu f(x+(\xi+1)h)-\partial^\nu f(x+\xi h)\big)\right| \\
&\leq  \max_{|\nu|=s-1} |\partial^\nu f|_{\beta,X^{+\delta}}\,\norm{h}^{\beta}\left| \sum_{|\nu|=s-1} \dfrac{(s-1)!}{\nu!} \prod_{j=1}^d |h_j|^{\nu_j} \right| \\
&=\max_{|\nu|=s-1} |\partial^\nu f|_{\beta,X^{+\delta}}\,\norm{h}^{\beta} \norm{h}_{\ell_1^d}^{s-1}\\
&\leq d^\frac{s-1}{2}\max_{|\nu|=s-1} |\partial^\nu f|_{\beta,X^{+\delta}}\,\norm{h}^\alpha
\end{align*}
using the definition of the H\"older seminorm (\ref{eqn:Holder_quasinorm}) in the first inequality, the multinomial theorem in the next step, and $\beta+s-1=\alpha$ in the last step. When $\alpha$ is an integer we analogously get 
\begin{align*}
|\Delta_1F^{(s-1)}(\xi)|\leq 2d^\frac{s-1}{2}\max_{|\nu|=s-1} \norm{\partial^\nu f|_{X^{+\delta}}}_\infty\norm{h}^\alpha.
\end{align*}
In both cases this implies
\begin{align*}
\omega_{s,L_2(\mu)}(f,t)=\sup_{\norm{h}\leq t}\norm{\Delta_h^s f}_{L_2(\mu)}\leq C t^{\alpha}
\end{align*}
for $t<\delta/(s-1)$. For $t\geq \delta/(s-1)$ we can estimate
\begin{align*}
t^{-\alpha}\omega_{s,L_2(\mu)}(f,t)&\leq \left( \dfrac{\delta}{s-1}\right)^{-\alpha} \sup_{\norm{h}\leq t} \norm{\Delta_h^s f}_{L_2(\mu)}\\
&\leq  \left( \dfrac{\delta}{s-1}\right)^{-\alpha}\sup_{\norm{h}\leq t} \sum_{j=0}^s \binom{s}{j} \norm{f(\cdot+jh)}_{L_2(\mu)} \\
&\leq  \left( \dfrac{\delta}{s-1}\right)^{-\alpha}2^s \mu(X)\norm{f}_\infty
\end{align*}
which gives us $|f|_{B_{2,\infty}^\alpha(\mu)}=\sup_{t>0} t^{-\alpha} \omega_{s,L_2(\mu)}(f,t)<\infty$.
\end{proof}
To prove Theorem \ref{thm:ls_oracle_inequality} we need a suitable function $f_0\in H_\gamma(X)$ bounding the approximation error. To this end, we first collect some facts on Gaussian RKHSs, which are a summary of \cite[Theorem 4.21, Lemma 4.45, and Proposition 4.46]{SteinwartChristmannSVMs}. By introducing the function $K_\gamma:\R^d\rightarrow\R$ defined by
\begin{align*}
K_\gamma(x):=\left( \dfrac{2}{\gamma\sqrt{\pi}}\right)^\frac{d}{2}\exp(-2\gamma^{-2}\norm{x}^2),\qquad x\in\R^d,
\end{align*}
the Gaussian RKHS $H_\gamma(X)$ can be characterized as the image of the convolution operator $L_2(\R^d)\rightarrow H_\gamma(X)$ defined by $ g\mapsto K_\gamma*g$. The $H_\gamma(X)$-norm can be computed by $\norm{f}_{H_\gamma(X)}=\inf\{\norm{g}_{L_2(\R^d)}:f=K_\gamma*g\}$. Furthermore, for $0<\gamma_1<\gamma_2<\infty$ the space $H_{\gamma_2}(X)$ is continuously embedded into $H_{\gamma_1}(X)$ with
\begin{align}
\norm{\id:H_{\gamma_2}(X)\rightarrow H_{\gamma_1}(X)}\leq \left(\frac{\gamma_2}{\gamma_1}\right)^\frac{d}{2}. \label{eqn:embedding_property}
\end{align}
We will further make use of integration in spherical coordinates \cite[Theorem 2.49]{Folland_RealAnalysis}. Namely, for $f\in L_1(\R^d)$ or $f\geq 0$ we have
\begin{align}\label{eqn:integration_spherical_coordinates}
\int_{\R^d} f(x)\,\d x = \int_0^\infty \int_{\mathbb{S}^{d-1}} f(r\omega)r^{d-1}\,\d\sigma(\omega)\d r,
\end{align}
where $\mathbb{S}^{d-1}=\{x\in\R^d:\norm{x}=1\}$ and $\sigma$ is the surface measure on $\mathbb{S}^{d-1}$. For radial functions $f$, that is $f(x)=g(\norm{x})$ Equation (\ref{eqn:integration_spherical_coordinates}) simplifies to
\begin{align}\label{eqn:integral_radial_function}
\int_{\R^d} f(x)\,\d x= \frac{2\pi^\frac{d}{2}}{\Gamma\left( \frac{d}{2}\right)} \int_0^\infty g(r)r^{d-1}\,\d r,
\end{align}
since $\sigma(\mathbb{S}^{d-1})=2\pi^{d/2}/\Gamma(d/2)$, see e.g.~\cite[Proposition 2.54]{Folland_RealAnalysis}. Using (\ref{eqn:integral_radial_function}) one can easily check, that
\begin{align}\label{eqn:integral_gauss_kernel}
\int_{\R^d} (\gamma\sqrt{\pi})^{-\frac{d}{2}}K_\gamma(x)\,\d x=\int_{\R^d} \gamma^{-d}\left( \frac{2}{\pi}\right)^\frac{d}{2}\exp\left( -2\gamma^{-2}\norm{x}^2 \right)\,\d x=1,
\end{align}
which we will rely on later. Finally, we define
\begin{align}
G:=\sum_{j=1}^s \binom{s}{j} (-1)^{1-j} (j\gamma\sqrt{\pi})^{-\frac{d}{2}} K_{j\gamma}. \label{eqn:convolution_kernel}
\end{align}
The following lemma bounds the approximation error of a Gaussian SVM using the least-squares loss.
\begin{lem}\label{lem:L2error}
For $f\in L_2(\R^d)$ with $|f|_{B_{2,\infty}^\alpha(\P_X)}<\infty$ we have
\begin{align*}
\norm{G*f-f}_{L_2(\P_X)}^2 \leq |f|_{B_{2,\infty}^\alpha(\P_X)}^2 2^{-\alpha}\left(\dfrac{\Gamma\left( \frac{\alpha+d}{2}\right)}{\Gamma\left( \frac{d}{2}\right)} \right)^2 \gamma^{2\alpha}.
\end{align*}
\end{lem}
\begin{proof}
We set $s=\lfloor\alpha\rfloor+1$ and compute
\begin{align*}
G*f(x)&=\int_{\R^d} \sum_{j=1}^s \binom{s}{j} (-1)^{1-j} (j\gamma)^{-d}\left( \dfrac{2}{\pi}\right)^\frac{d}{2}\exp\left(-2(j\gamma)^{-2}\norm{y}^2\right)f(x+y)\d y\\
&=\int_{\R^d} \sum_{j=1}^s\binom{s}{j}(-1)^{1-j} \gamma^{-d}\left( \dfrac{2}{\pi}\right)^\frac{d}{2} \exp\left( -2\gamma^{-2}\norm{h}^2\right)f(x+jh)\d h.
\end{align*}
Using equation (\ref{eqn:integral_gauss_kernel}) this implies
\begin{align*}
G*f(x)-f(x) =&\,G*f(x)-\int_{\R^d}\gamma^{-d}\left( \dfrac{2}{\pi}\right)^\frac{d}{2}\exp\left( -2\gamma^{-2}\norm{h}^2\right)f(x)\d h \\
=&\int_{\R^d} \sum_{j=1}^s\binom{s}{j}(-1)^{1-j} \gamma^{-d}\left( \dfrac{2}{\pi}\right)^\frac{d}{2} \exp\left( -2\gamma^{-2}\norm{h}^2\right)f(x+jh)\d h \\
&-\int_{\R^d}\gamma^{-d}\left( \dfrac{2}{\pi}\right)^\frac{d}{2}\exp\left( -2\gamma^{-2}\norm{h}^2\right)f(x)\d h \\
=&\int_{\R^d} (-1)^{1-s}\left( \dfrac{2}{\gamma^2\pi}\right)^\frac{d}{2}\exp\left( -2\gamma^{-2}\norm{h}^2\right)\Delta_h^sf(x)\d h.
\end{align*}
With this identity we can bound our desired $L_2(\P_X)$-norm by
\begin{align*}
&\norm{G*f-f}_{L_2(\P_X)}^2 \\
=& \int_{\R^d}\left( \int_{\R^d} \left( \dfrac{2}{\gamma^2\pi}\right)^\frac{d}{2}\exp\left( -2\gamma^{-2}\norm{h}^2\right)\Delta_h^sf(x)\d h\right)^2\d\P_X(x)\\
\leq& \left(\int_{\R^d} \left(\int_{\R^d}\left( \dfrac{2}{\gamma^2\pi}\right)^d \left( \exp\left( -2\gamma^{-2}\norm{h}^2\right)\Delta_h^sf(x)\right)^2\d\P_X(x)\right)^\frac{1}{2}\d h\right)^2 \\
=& \left( \int_{\R^d} \left( \dfrac{2}{\gamma^2\pi}\right)^\frac{d}{2}\exp\left(-2\gamma^{-2}\norm{h}^2\right) \norm{\Delta_h^sf}_{L_2(\P_X)}\d h\right)^2
\end{align*}
using Minkowski's integral inequality. With our assumptions on $f$ we can further bound this by
\begin{align*}
\norm{G*f-f}_{L_2(\P_X)}^2&\leq\left( \int_{\R^d} \left( \dfrac{2}{\gamma^2\pi}\right)^\frac{d}{2}\exp\left( -\gamma^{-2}\norm{h}^2\right)\omega_{s,L_2(\P_X)}(f,\norm{h})\d h \right)^2 \\
&\leq  |f|_{B_{2,\infty}^\alpha(\P_X)}^2 \left(\dfrac{2}{\gamma^2\pi}\right)^d \left( \int_{\R^d}\exp\left( -2\gamma^{-2}\norm{h}^2\right)\norm{h}^\alpha\d h \right)^2
\end{align*}
which leaves us with computing the integral in the last step. This is done using spherical coordinates, which gives us
\begin{align*}
\int_{\R^d}\exp\left( -2\gamma^{-2}\norm{h}^2\right)\norm{h}^\alpha\d h&=\dfrac{2\pi^\frac{d}{2}}{\Gamma\left(\frac{d}{2}\right)}\int_0^\infty \exp\left( -2\gamma^{-2}r^2\right)r^{\alpha+d-1}\d r \\
&=\dfrac{2\pi^\frac{d}{2}}{\Gamma\left(\frac{d}{2}\right)} \int_0^\infty \frac{1}{2}\left(\dfrac{\gamma}{\sqrt{2}}\right)^{\alpha+d}\e^{-u}\,u^{\frac{\alpha+d}{2}-1}\,\d u \\
&= \dfrac{\pi^\frac{d}{2}}{\Gamma\left( \frac{d}{2}\right)} \left(\dfrac{\gamma}{\sqrt{2}}\right)^{\alpha+d}\Gamma\left(\dfrac{\alpha+d}{2}\right).
\end{align*}
Combining these considerations we get
\begin{align*}
\norm{G*f-f}_{L_2(\P_X)}^2 \leq |f|_{B_{2,\infty}^\alpha(\P_X)}^2 2^{-\alpha}\left(\dfrac{\Gamma\left( \frac{\alpha+d}{2}\right)}{\Gamma\left( \frac{d}{2}\right)} \right)^2 \gamma^{2\alpha}.
\end{align*}
\end{proof}
The following lemma bounds the regularization term.
\begin{lem}\label{lem:regularization_term}
For $f\in L_2(\R^d)$ we have $\norm{G*f}_{H_\gamma(X)}\leq (\gamma\sqrt{\pi})^{-\frac{d}{2}}2^s\norm{f}_{L_2(\R^d)}$.
\end{lem}
\begin{proof}
Because of the embedding property (\ref{eqn:embedding_property}) we have
\begin{align*}
\norm{G*f}_{H_\gamma(X)}&\leq \sum_{j=1}^s \binom{s}{j} (j\gamma\sqrt{\pi})^{-\frac{d}{2}}\norm{K_{j\gamma}*f}_{H_\gamma(X)}\\
&\leq  \sum_{j=1}^s\binom{s}{j} (\gamma\sqrt{\pi})^{-\frac{d}{2}}\norm{K_{j\gamma}*f}_{H_{j\gamma}(X)}\\
&\leq (\gamma\sqrt{\pi})^{-\frac{d}{2}}\norm{f}_{L_2(\R^d)}\sum_{j=1}^s\binom{s}{j}\\
&\leq (\gamma\sqrt{\pi})^{-\frac{d}{2}}2^s\norm{f}_{L_2(\R^d)}.
\end{align*}
\end{proof}
\begin{proof}[Proof of Theorem \ref{thm:ls_oracle_inequality}]
For $Y=[-M,M]$ the least-squares loss satisfies the supremum/ variance bound for the constants $B=4M^2, V=16M^2$ and $\vartheta=1$ by \cite[Example 7.3]{SteinwartChristmannSVMs}. Theorem \ref{thm:general_oracle_inequality} therefore gives us for $\lambda>0, \gamma\in (0,1)$, and $p\in(0,1/2]$
\begin{align*}
\Rk_{L,\P}(\wideparen{f}_{D,\lambda,\gamma})-\Rk_{L,\P}^*\leq &\,9(\lambda\norm{f_0}_{H_\gamma(X)}^2+\Rk_{L,\P}(f_0)-\Rk_{L,\P}^*) \\
&+ C_\P K p^{-d-1}\gamma^{-\varrho}\lambda^{-p}n^{-1}+\frac{(3456M^2+15B_0)\tau}{n}
\end{align*}
with probability $\P^n$ not less than $1-3\e^{-\tau}$. To bound the approximation error we set $f_0:=G\ast f_{L,\P}^*$, where $G$ is defined by (\ref{eqn:convolution_kernel}) for $s=\lfloor\alpha\rfloor+1$. First we determine $B_0$, i.e.~a bound on $\sup_{(x,y)\in X\times Y}|y-f_0(x)|^2$. By Young's convolution inequality we have
\begin{align*}
\norm{f_0}_{L_\infty(\R^d)}=\norm{G*f_{L,\P}^*}_{L_\infty(\R^d)}\leq \norm{f_{L,\P}^*}_{L_\infty(\R^d)}\cdot\norm{G}_{L_1(\R^d)}
\end{align*}
and by using Equation (\ref{eqn:integral_gauss_kernel}) we get
\begin{align*}
\norm{G}_{L_1(\R^d)}\leq \sum_{j=1}^s \binom{s}{j} (j\gamma\sqrt{\pi})^{-\frac{d}{2}} \norm{K_{j\gamma}}_{L_1(\R^d)}=\sum_{j=1}^s\binom{s}{j}\leq 2^s.
\end{align*}
Consequently, we get
\begin{align*}
\sup_{(x,y)\in X\times Y}|f_0(x)-y|^2\leq \left(2^s\norm{f_{L,\P}^*}_{L_\infty(\R^d)}+M\right)^2,
\end{align*}
i.e.~we can set
\begin{align*}
B_0:=\max\left\{\left(2^s\norm{f_{L,\P}^*}_{L_\infty(\R^d)}+M\right)^2,4M^2 \right\}.
\end{align*}
Using Lemma \ref{lem:L2error} and Lemma \ref{lem:regularization_term} with $s=\lfloor\alpha\rfloor+1$ we can bound the regularization term and the excess risk as stated in the theorem. To determine a bound on $C_\P$ first note that $|\Lls|_{M,1}\leq 4M$. Some calculations then show $C_\P\leq \max\{16M^2,1\}\max\{\Cdim,4M^2\}$. Finally, the results follows by setting $p=\log 2/(2\log n)\leq 1/2$.
\end{proof}
\begin{proof}[Proof of Corollary \ref{cor:LS_learning_rates}]
Theorem \ref{thm:ls_oracle_inequality} gives us
\begin{align*}
\Rk_{L,\P}(\wideparen{f}_{D,\lambda,\gamma})-\Rk_{L,\P}^* \leq C\left( \lambda\gamma^{-d}+\gamma^{2\alpha}+\lambda^{-1/\log n}\gamma^{-\varrho}n^{-1}\log^{d+1}n+\frac{\tau}{n}\right)
\end{align*}
with probability $\P^n$ not less than $1-3\e^{-\tau}$ for all $n\in\N$ and a constant $C$ only depending on $\Cdim,C_{1,2,3}$ and $M$. With the choices of $\lambda_n$ and $\gamma_n$ as stated in the corollary we get
\begin{align*}
\Rk_{L,\P}(\wideparen{f}_{D,\lambda,\gamma})-\Rk_{L,\P}^* &\leq C\left( n^{-b}n^{\frac{d}{2\alpha+\varrho}} +n^{-\frac{2\alpha}{2\alpha+\varrho}} +\e^b n^{-\frac{2\alpha}{2\alpha+\varrho}}\log^{d+1}n+\frac{\tau}{n} \right) \\
&\leq C\left(2n^{-\frac{2\alpha}{2\alpha+\varrho}}+\e^b n^{-\frac{2\alpha}{2\alpha+\varrho}}\log^{d+1}n+\frac{\tau}{n}\right)
\end{align*}
with probability $\P^n$ not less than $1-3\e^{-\tau}$ for all $n\in\N$. A substitution of $\tau$ then easily proves the assertion.
\end{proof}
\begin{proof}[Proof of Theorem \ref{thm:LS_adaptive_rates}]
We define $\gamma_n:=n^{-1/(2\alpha+\varrho)} $ and $ \lambda_n:=n^{-(2\alpha+d)/(2\alpha+\varrho)}$. By \cite[Theorem 7.2]{SteinwartChristmannSVMs}, which states an oracle inequality for empirical risk minimization over finite hypothesis sets, we have
\begin{align*}
\Rk_{L,\P}(\wideparen{f}_{D_1,\lambda_{D_2},\gamma_{D_2}})-\Rk_{L,\P}^*\leq&\, 6 \min_{(\lambda,\gamma)\in \Lambda_n\times \Gamma_n} \left( \Rk_{L,\P}(\wideparen{f}_{D_1,\lambda,\gamma})-\Rk_{L,\P}^* \right) \\
&+\frac{512M^2(\tau+\log(1+|\Lambda_n\times\Gamma_n|))}{n-m} \\
\leq &\,6\left( \Rk_{L,\P}(\wideparen{f}_{D_1,\lambda_0,\gamma_0})-\Rk_{L,\P}^* \right)\\
&+\frac{2048M^2(\tau+\log(1+|\Lambda_n\times\Gamma_n|))}{n}
\end{align*}
with probability $\P^{n-m}$ not less than $1-\e^{-\tau}$, where in the last step we picked $\gamma_0:=n^{-a}\in\Gamma_n$ and $\lambda_0:=n^{-b}\in\Lambda_n$ for values $a$ and $b$, which we will specify in a moment. An application of Theorem \ref{thm:ls_oracle_inequality} gives us
\begin{align*}
\Rk_{L,\P}(\wideparen{f}_{D_1,\lambda_0,\gamma_0})-\Rk_{L,\P}^* &\leq C\left( \lambda_0\gamma_0^{-d}+\gamma_0^{2\alpha}+b^{d+1}\gamma_0^{-\varrho}m^{-1}\log^{d+1}n+\frac{\tau}{m} \right) \\
&\leq C \left(  \lambda_0\gamma_0^{-d}+\gamma_0^{2\alpha}+2b^{d+1}\gamma_0^{-\varrho}n^{-1}\log^{d+1}n+\frac{2\tau}{n} \right)
\end{align*}
with probability $\P^m$ not less than $1-3\e^{-\tau}$. Now let $\lambda_0=n^{-d}$ and let $a\in A_n$ satisfy $1/(2\alpha+\varrho)\leq a \leq 1/(2\alpha+\varrho)+1/\log n$, which implies
\begin{align*}
\Rk_{L,\P}(\wideparen{f}_{D_1,\lambda_0,\gamma_0})-\Rk_{L,\P}^*\leq&\,C \left( \e^d \lambda_n\gamma_n^{-d}+\gamma_n^{2\alpha}+2d^{d+1}\e^\varrho\gamma_n^{-\varrho} n^{-1}\log^{d+1}n+\frac{2\tau}{n} \right) \\
=&\, C\left( \e^d n^{-\frac{2\alpha}{2\alpha+\varrho}}+n^{-\frac{2\alpha}{2\alpha+\varrho}}+2d^{d+1}\e^{\varrho}n^{-\frac{2\alpha}{2\alpha+\varrho}}\log^{d+1}n+\frac{2\tau}{n} \right)
\end{align*}
with probability $\P^m$ not less than $1-3\e^{-\tau}$. Combining these inequalities and using that $|\Lambda_n\times \Gamma_n|\in\O(\log^2 n)$ we get
\begin{align*}
\Rk_{L,\P}(\wideparen{f}_{D_1,\lambda_{D_2},\gamma_{D_2}})-\Rk_{L,\P}^*\leq c_1 \left( n^{-\frac{2\alpha}{2\alpha+\varrho}}\log^{d+1}n +\frac{\tau}{n}\right)+c_2\left( \frac{\tau}{n}+\frac{\log n}{n} \right)
\end{align*}
with probability $\P^n$ not less than $(1-\e^{-\tau})(1-3\e^{-\tau})\geq1-4\e^{-\tau}$ for all $n>2$.
\end{proof}

\subsection{Proofs related to Section \ref{sec:learning_rates_classification}}\label{sec:proofs_class}
\begin{proof}[Proof of Theorem \ref{thm:class_oracle_inequality}]
The distribution $\P$ obviously satisfies the supremum bound for $B=2$ and by \cite[Theorem 8.24]{SteinwartChristmannSVMs} the variance bound is satisfied for $V=6C_*^{q/(q+1)}$ and $\vartheta=q/(q+1)$, where $C_*$ is the constant from Assumption \ref{as:TNC}. Theorem \ref{thm:general_oracle_inequality} then gives us for $\lambda>0,\gamma\in(0,1)$, and $p\in(0,1/2]$
\begin{align*}
\Rk_{L,\P}(\wideparen{f}_{D,\lambda,\gamma})-\Rk_{L,\P}^*\leq&\, 9\left(  \lambda\norm{f_0}_{H_\gamma(X)}^2+\Rk_{L,\P}(f_0)-\Rk_{L,\P}^*\right)\\
&+C_\P K\left( \frac{p^{-d-1}\lambda^{-p}\gamma^{-\varrho}}{n} \right)^\frac{q+1}{q+2-p}+3\left(\frac{432C_*^{\frac{q}{q+1}}\tau}{n}\right)^\frac{q+1}{q+2}+\frac{15B_0\tau}{n}
\end{align*}
with probability not less than $1-3\e^{-\tau}$. In \cite[Theorem 8.18]{SteinwartChristmannSVMs} a function $f_0\in H_\gamma(X)$ with $\norm{f_0}_\infty\leq 1$ and
\begin{align*}
\lambda\norm{f_0}_{H_\gamma(X)}^2+\Rk_{L,\P}(f_0)-\Rk_{L,\P}^*\leq c_1\lambda\gamma^{-d}+c_2C_{**}\gamma^\beta,
\end{align*}
is constructed, where $c_1=3^d/\Gamma(d/2+1)$, $c_2=2^{1-\beta/2}\Gamma((\beta+d)/2)/\Gamma(d/2)$ and $C_{**}$ is the constant from Assumption \ref{as:MNE}. Further, we have $|\Lhinge|_{1,1}=1$ and since $\norm{f}_\infty\leq 1$ we can choose $B_0=2$. Simple calculations yield $C_\P\leq \max\{4,C_*^{q/(q+1)}\}\max\{2,\Cdim\}$. Finally, choosing $p=\log 2/(2\log n)\leq 1/2$ and some simple estimates proves the result.
\end{proof}
As the proof of Corollary \ref{cor:class_learning_rates} merely consists of plugging in the specified values, we will skip it at this point.

\begin{proof}[Proof of Theorem \ref{thm:class_adaptive_rates}]
Recall that our optimal choice for $\gamma_n$ and $\lambda_n$ was given by $\gamma_n=n^{-a}$ and $\lambda_n=n^{-b}$ with
\begin{align*}
a= \frac{q+1}{\beta(q+2)+\varrho(q+1)}  \quad \text{ and } \quad  b\geq \frac{(d+\beta)(q+1)}{\beta(q+2)+\varrho(q+1)}.
\end{align*}
Now note that for $\varrho\geq 1$ we have $a\leq 1$ and the choice $b= d$ is admissible. That is, by construction $\Gamma_n$ and $\Lambda_n$ cover a possible choice of $\gamma_n$ and $\lambda_n$, which achieve optimal rates. The statement can now be proven exactly as in the proof of Theorem \ref{thm:LS_adaptive_rates} by using Theorem \ref{thm:class_oracle_inequality}.
\end{proof}
%
%
\end{appendix}
\section*{Acknowledgements}
The authors thank the International Max Planck Research School for Intelligent Systems (IMPRS-IS) for supporting Thomas Hamm. Ingo Steinwart was supported by the German Research Foundation under  DFG Grant STE 1074/4-1.

\bibliography{bib}

\end{document}